\documentclass[11pt]{article}
\usepackage{graphicx,epic,eepic}
\usepackage{amsmath, amsthm, verbatim, amsfonts, amssymb, color}
\usepackage{mathrsfs}
\usepackage{dsfont}
\usepackage[a4paper,margin=3cm]{geometry}
\usepackage{epstopdf}
\usepackage{nicefrac}

\numberwithin{equation}{section}
\renewcommand\labelenumi{\textup{\alph{enumi})}}
\renewcommand\theenumi\labelenumi

\usepackage{marginnote}

\usepackage{amssymb,amsmath,amsthm,mathdesign}
\usepackage{pdfsync}
\usepackage{color}
\usepackage[colorlinks]{hyperref}

\newtheorem{theorem}{Theorem}[section]
\newtheorem{lemma}[theorem]{Lemma}

\newtheorem{proposition}[theorem]{Proposition}
\newtheorem{corollary}[theorem]{Corollary}
\newtheorem{assumption}[theorem]{Assumption}
\newtheorem{remark}[theorem]{Remark}
\newtheorem{example}[theorem]{Example}
\numberwithin{equation}{section}

\newcommand{\be}{\begin{equation}}
\newcommand{\ee}{\end{equation}}

\newcommand{\bes}{\begin{equation*}}
\newcommand{\ees}{\end{equation*}}

\def\E{\bE}

\def\P{\bP} 

\newcommand{\N}{\mathbf{varrho}}

\newcommand{\R}{\mathbf{R}}

\renewcommand{\d}{{\rm d}}

\renewcommand{\geq}{\geqslant}
\renewcommand{\leq}{\leqslant}
\renewcommand{\ge}{\geqslant}
\renewcommand{\le}{\leqslant}

\renewcommand{\P}{\mathrm{P}}

\def\R{{\mathbb R}}
\def\N{{\mathbb N}}

\def\E{{\mathbb E}}
\def\P{{\mathbb P}}

\def\m1{\mathbf{1}}

   \DeclareMathOperator{\Cov}{Cov}

\allowdisplaybreaks[1]

\allowdisplaybreaks[1]

\title{\textbf{Level of noises and long time behavior of the solution for
space-time fractional SPDE in bounded domains}}
\author{	\textbf{Jebessa B. Mijena}\\ Department of Mathematics, Georgia College \& State University,
 GA 31061, USA\\
 Email: jebessa.mijena@gcsu.edu
\and \textbf{Erkan Nane}\\ Department of  Mathematics and Statistics, Auburn University, Alabama 36849, USA \\
Email: ezn0001@auburn.edu
\and \textbf{Alemayehu G. Negash} \\ Department of  Mathematics and Statistics, Auburn University, Alabama 36849, USA\\
Email: agn0008@auburn.edu
}
\date{\today}

\begin{document}
\maketitle

\begin{abstract}
In this paper we study the long time behavior of the solution to a certain class of  space-time fractional stochastic equations with respect to the level $\lambda$ of a noise  and show how the choice of the order $\beta \in (0, \,1)$ of the fractional time derivative affects the growth and decay behavior of their  solution. We consider both the cases of white noise and colored noise. Our results extend the main results in Foondun \cite{Foondun2} to fractional Laplacian as well as higher dimensional cases.

\end{abstract}
{\bf Keywords:} Space-time fractional SPDE, space-time white noise, space colored noise, moment bounds in bounded domains.

\maketitle

\noindent
\newpage
\section{Introduction and Statement of the Main results}
In 1855, Adolf Fick \cite{Fick} developed his now-famous principle governing the transport of mass through diffusive media.  In his second law, Fick showed that $u(t,\,\text{x})$ obeys the classical diffusion equation $\partial_t u = D \partial_{xx}u$ in the spatial dimension, which estimates how the concentration $u(t,\,\text{x})$ of a diffusive substance varies with space and time, where $D$ is the diffusion coefficient.
The space-time fractional diffusion equation which is obtained when integer-order derivative operators in space and time are replaced by fractional counterparts (in Caputo or Riemann-Liouville sense) has
been recently treated by several  authors (see, for example, Mijena and Nane \cite{Nane1}, Saichev and Zaslavsky  \cite{Saichev}, and  Mainardi, Luchko, and  Pagnini\cite{Mainardi}). The typical form of the space-time fractional equation is $\partial_t^\beta u = -(-\Delta)^{\alpha/2}u$, where $\partial_t^\beta$ is the Caputo fractional derivative with $\beta \in (0,\,1)$, $\alpha \in (0,\,2)$ and $\Delta = \sum_{i=1}^{d}\partial_{x_i}^2$ is the Laplacian. These equations can be used to model anomalous diffusion processes or diffusion processes in non-homogeneous media with random fractal structures (see, for instance,  Meerschaert \textit{et al.} \cite{Meerschaert2}, Meerschaert, Nane, and Vellaisamy \cite{Meerschaert4}, Baeumer, Luks, and Meerschaert \cite{Baeumer1}, Chen, Kim and Kim \cite{Chen2} and the references therein). Since these space-time fractional equations depend on the fractional parameters $\beta$ and $\alpha$, in \cite{Dang} the importance of the continuity of their solutions with respect to these parameters is discussed. For example, if the partial derivative in time $\partial_t$ in the classical heat equation $\partial_t u = -\Delta u$ is substituted with fractional derivatives $\partial_t^{\beta}$ for $0 <\beta <1$, the processes explains the sticking and trapping behavior of particle, while if  the Laplacian $\Delta$  is replaced with fractional power $-(-\Delta)^{\alpha/2}$ for $0<\alpha <2$, it describes long particle jumps (see \cite{Ahn}).
In \cite{Nane1}, Mijena and Nane have recently introduced time fractional SPDEs, which can be utilized to represent phenomena with random effects and thermal memory.\\
Consider the following space-time fractional equation with Dirichlet boundary conditions (see \eqref{JNA-3b} below for a representation of the solution).
\begin{equation}\label{JNA01}
\left \{\begin{aligned}
&\partial^\beta_t u_t(\text{x})=
-(-\Delta)^{\alpha/2} u(\text{x}),\,\, \text{x} \in B, \, t>0,\\
&\,u_t(x)=0, \,\, \text{x} \in B^C, \, \, t>0,
\end{aligned}\right.
\end{equation}
where $\alpha\in (0,\,2)$ and $\beta\in (0,\,1)$.
The fractional time derivative is the Caputo derivative which first appeared in \cite{Caputo} and is defined by
\begin{equation}\label{JNA02}
\partial^\beta_t u_t(\text{x})=\frac{1}{\Gamma(1-\beta)}\int_0^t \frac{\partial
u_r(\text{x})}{\partial r}\frac{\d r}{(t-r)^\beta}.
\end{equation}
If $u_0(\text{x})$ denotes the initial condition to equation \eqref{JNA01}, then the solution can be written as
$$u_t(\text{x})=\int_BG_B^{(\beta)}(t,\,x,\,y)u_0(\text{y}) \, dy,$$
where $G_B^{(\beta)}(t,\,x,\,y)$ is the space-time fractional heat kernel defined in \eqref{JNA-4}.\\
Now consider
 \begin{align}\label{JNA03}
 \partial^\beta_t u_t(\text{x})&=
 -(-\Delta)^{\alpha/2} u(\text{x}) + f(t,\text{x}),
\end{align}
with the same initial condition $u_0(\text{x})$ and $f(t,\text{x})$ is some nice function.
To get the correct version of \eqref{JNA03} we will make use of \cite{Umarov1, Umarov2, Umarov3}. The fractional Duhamel principle implies that the mild solution to \eqref{JNA03} for $t>0$ is given by
\begin{equation}\label{JNA04}
u_t(\text{x})=\int_BG_B^{(\beta)}(t,\,\text{x},\,\text{y})u_0(\text{y}) \, \d \text{y} + \int_0^t\int_BG_B^{(\beta)}(t-s,\,\text{x},\,\text{y}) \partial_s^{1-\beta}f(s,\text{y})\, \d \text{y} \,\d s.
\end{equation}
Using the fractional order integral $I^{\gamma}_t$ defined by
\begin{equation*}
I^{\gamma}_{t}f(t):=\frac{1}{\Gamma(\gamma)}\int_{0}^{t}(t-\tau)^{\gamma -1}f(\tau)\d\tau,
\end{equation*}
and the property
$$\partial^\beta_t I^{\beta}_{t}g(\text{t})=g(\text{t}),$$
for every $\beta \in (0,\,1)$, and $g \in L_{\infty} (\R_+)$ or $g \in C(\R_+)$, then by the Duhamel’s principle, the mild solution to \eqref{JNA03} where the force is $f(t,\, x) = I_t^{1-\beta} g(t,\,x)$, will be given by
\begin{align}\label{JNA05}
 u_t(\text{x})&=\int_BG_B^{(\beta)}(t,\,x,\,y)u_0(\text{y}) \, dy + \int_0^t\int_BG_B^{(\beta)}(t-s,\,x,\,y) \partial_s^{1-\beta}(I^{1-\beta}_{t}g(s,\,\text{y}))\, dy\,ds \nonumber \\
&=\int_BG_B^{(\beta)}(t,\,x,\,y)u_0(\text{y}) \, dy + \int_0^t\int_BG_B^{(\beta)}(t-s,\,x,\,y) g(s,\,\text{y})\, dy\,ds.
\end{align}
Recently, Mijena and Nane \cite{Nane1}, Foondun, Mijena, and Nane \cite{FMNane}, and Foondun \cite{Foondun2} considered the following time fractional stochastic heat equation on the interval $(0,\,L)$ with Dirichlet boundary condition:
\begin{equation}\label{JNA001}
\left \{\begin{aligned}
&\partial^\beta_t u_t(x)=\frac{1}{2}\partial _{xx}u_t(x)+I_t^{1-\beta}[\lambda \sigma(u_t(x))\dot{W}(t,\,x)]\;\; \text{for}\;\; 0<x<L\;\;\text{and}\; \;t>0\\
&u_t(0)=u_t(L)=0 \quad \text{for}\quad t>0,
\end{aligned}\right.
\end{equation}
where the initial condition $u_0:[0,L] \rightarrow \R_+$ is non-random and non-negative bounded function which is strictly positive on a set of positive measures in $[0,\,L]$. $\dot{W}$ denotes a space-time white noise and $\sigma:\R\rightarrow \R$ is a globally Lipschitz function satisfying $l_\sigma|x|\leq  |\sigma(x)| \leq L_\sigma|x|$ where $l_\sigma$ and $L_\sigma$ are positive constants. $\lambda$ is a positive parameter known as {\it the level of the noise} and will play a significant part in this paper.\\
Using Walsh \cite{Walsh}, we define the mild solution to \eqref{JNA001} as the random field
$u = \{u_t(x)\}_{t>0,x \in B}$
satisfying
\begin{equation}\label{JNA001a}
u_t(x)= (\mathcal{P}_B^{\beta}u_0)_t(x)+ \lambda \int_0^L\int_0^t p_B^{\beta}(t-s,\,x,\,y)\sigma(u_s(y))W(\d s\,,\d y),
\end{equation}
where $p_B^{\beta}(t,x,y)$ denotes the probability density function of the time-changed killed Brownian motion upon exiting the domain $[0,\,L]$ associated with the fractional time operator, and
\begin{equation*}
(\mathcal{P}_B^{\beta} u_0)_t(x):=\int_0^L p_B^{\beta}(t,\,x,\,y) u_0(y)\d y.
\end{equation*}

In Foondun and Nualart \cite{Foondun1} and Foondun, Guerngar, and Nane \cite{FGNane}, the authors looked at the behavior of the solution to equation \eqref{JNA001} for small and large $\lambda$ when $\boldsymbol\beta =1$. They showed that
if $\lambda$ is large enough, the second moment of the solution $u_t$ grows exponentially fast; while if $\lambda$ is small, the second moment of the solution $u_t$ eventually decays exponentially.
Nualart \cite{Nualart1} and Xie \cite{Xie} have used precise heat kernel estimates to sharpen the results in \cite{Foondun1}. However, in \cite{Foondun2} Foondun  has shown that a more complicated situation will occur instead of such phase transition if fractional time derivative replaces the usual time derivative. That is, for any fixed $\beta \in (0,\, 1)$, the long time behavior of the supremum of the second moment of the solution to equation \eqref{JNA001} behaves differently by considering the cases when $\beta \in (0,\, \frac{1}{2})$ and $\beta \in (\frac{1}{2}, \, 1)$ separately, and the reasons for considering these cases are also explained. For more details about the interpretation of the results, one can refer \cite{Foondun2}.  These findings are interesting from an application standpoint since fractional time derivatives are commonly used in the modeling of various systems with memory. Therefore, it is very important to realize that the use of such derivatives can result in considerable change in the qualitative properties of the solution. The main aim of this work is to investigate the long time behavior of the solution to \eqref{JNA-8} with respect to the level of the noise $\lambda$. Our work extend the main results in Foondun \cite{Foondun2} to fractional Laplacian as well as higher dimensional cases.

\newpage

\newpage
Consider the following stochastic heat equation on a regular bounded domain $B$ in $\R^d$, $d\geq 1$ with Dirichlet boundary condition:
\begin{equation}\label{JNA-8}
\left \{\begin{aligned}
&\partial^\beta_t u_t(\text{x})= -\left(-\Delta \right)^{\frac{\alpha}{2}} u_t(\text{x})+I_t^{1-\beta}[\lambda \sigma(u_t(\text{x}))\dot{W}(t,\,\text{x})]\;\; \text{for}\;\; \text{\text{x}} \in B\;\;\text{and}\; \;t>0\\
&u_t(\text{x})=0 \quad \text{for} \quad \text{x} \notin B\;\;\text{and}\; \;t>0,
\end{aligned}\right.
\end{equation}
and the initial condition $u_0:B \rightarrow \R_+$ is  a non-random measurable and bounded function that has support with positive measure inside B. The operator $-\left(-\Delta \right)^{\frac{\alpha}{2}}$, where $0< \alpha \leq 2$, is the $L^2$-generator of a symmetric $\alpha$-stable process $X_t^B$ killed when exiting $B$. $\dot{W}$ denotes a space-time white noise and $\sigma:\R\rightarrow \R$ is a globally Lipschitz function satisfying $l_\sigma|\text{x}|\leq  |\sigma(\text{x})| \leq L_\sigma|\text{x}|$ where $l_\sigma$ and $L_\sigma$ are positive constants. The positive parameter $\lambda$ is called {\it the level of the noise}.\\
$\dot{W}(t,\,\text{x})$ is a space-time white noise with $\text{x} \in B$, which is assumed to be adapted with respect to a filtered probability space $(\Omega,\, \mathcal{F}, \,\mathcal{F}_t, \,\P)$, where $\mathcal{F}$ is complete and the filtration $\{\mathcal{F}_t,\, t \geq 0\}$ is right continuous.\\
$\dot{W}(t,\,\text{x})$ is a generalized process with covariance given by
$$\E \left [\dot{W}(t,\,\text{x})\dot{W}(s,\,\text{y})\right]=\delta(t-s)\delta(\text{x}-\text{y}).$$
That is, $W(f)$ is a random field indexed by function $f \in L^2((0,\,\infty) \times B)$ and for all $f,g \in L^2((0,\,\infty) \times B)$, we have
$$\E \left [W(f)W(g)\right ]=\int_0^\infty \int_B f(t,\text{x}) g(t,\text{x})\,d\text{x}\,dt.$$
Hence $W(f)$ can be represented
$$W(f)=\int_0^\infty \int_B f(t,\text{x}) W(\,d\text{x}\,dt).$$
Note that $W(f)$ is $\mathcal{F}_t$-measurable whenever $f$ is supported on $[0,\,t]\times B$.\\
The Walsh-Dalang Integrals \cite{Dalang, Walsh} that is used in equation \eqref{JNA-9} is defined as follows. We use the Brownian Filtration $\{\mathcal{F}_t\}$ and the Walsh-Dalang
integrals defined as follows:
\begin{itemize}
\item $(t, x) \rightarrow \Phi_t(x)$ is an elementary random field when $\exists 0 \leq a<b$ and
an $\mathcal{F}_a$-measurable $X \in L^2(\Omega)$ and $\phi \in L^2(B)$ such that
$$\Phi_t(x) = X1_{[a,\,b]}(t)\phi(x)\ \ \, t > 0, x \in B.$$
\item If $h = h_t(x)$ is non-random and $\Phi$ is elementary, then
$$\int h\Phi dW :=X \int_{(a,\,b)\times B}h_t(x)\phi(x) W(dtdx).$$
\item The stochastic integral is Wiener’s, and it is well defined iff
$h_t(x)\phi(x) \in L^2([a,\,b] \times B)$.
\item We have Walsh isometry,
\begin{equation}
 \E \left (\left|\int h\Phi dW\right|^2\right) = \int_0^{\infty} ds \int_B dy [h_s(y)]^2\E(|\Phi_s(y)|^2).
\end{equation}
\end{itemize}
We can make sense of equation \eqref{JNA-8} using Walsh theory \cite{Walsh} again by using the integral equation below.
\begin{equation}\label{JNA-9}
u_t(x)=
(\mathcal{G}_B^{(\beta)} u_0)_t(\text{x})+ \lambda \int_B\int_0^t G_B^{(\beta)}(t-s,\,\text{x},\,\text{y})\sigma(u_s(\text{y}))W(\d s\,,\d \text{y}),
\end{equation}
where  $G_B^{(\beta)} (t,\text{x},\text{y})$ denotes the heat kernel of the space-time fractional diffusion equation with Dirichlet boundary conditions in \eqref{JNA01}, and
\begin{equation*}
(\mathcal{G}_B^{(\beta)} u_0)_t(x):=\int_B G_B^{(\beta)}(t,\,\text{x},\,\text{y}) u_0(\text{y})\,\d \text{y}.
\end{equation*}
If $d<(2 \wedge \beta^{-1})\alpha$, the proof for the existence of a unique random-field solution of \eqref{JNA-8} satisfying
$$\sup_{x \in B} \E|u_t(x)|^2 \leq c_1e^{c_2t\lambda^{\frac{2\alpha}{\alpha-\beta d}}}$$
for all $t>0$ can be found in Foondun \textit{et al.} \cite{FMNane}.\\
Next we state our main results.\\
We first note that when $\alpha = 2$ and $B=(0,\,1)$, \eqref{JNA-8} becomes \eqref{JNA001}. Thus,  our new results are  consistent with that obtained in Foondun \cite{Foondun2}
\begin{theorem}\label{Thrm-1}
Suppose that $d<(2 \wedge \beta ^{-1})\alpha$. Let $u_t$ denote the unique solution to \eqref{JNA-9}.  Then the second moment of $u_t$ cannot decay exponentially fast, regardless of what $\lambda$ is. Indeed, if we further suppose that $\beta\in (0,\,\frac{1}{2}]$, then as $t$ gets large, $\sup_{x\in B} \E|u_t(x)|^2 $ grows exponentially fast for any $\lambda.$
\end{theorem}
\begin{theorem}\label{Thrm-2} In Theorem \ref{Thrm-1} if $\beta\in (\frac{1}{2},\,1)$, then there exist a strictly positive real number $\lambda_u$ such that for all $\lambda>\lambda_u$,  $\sup_{x\in B} \E|u_t(x)|^2 $ grows exponentially fast as time gets large.
\end{theorem}
Since the Mittage-Leffler function $E_{\beta}(-t^\beta)$ (Gorenfeo \textit{et al.}\cite{Gorenflo}) behaves as a stretched exponential for $t\rightarrow 0$;
\begin{equation*}
    E_{\beta}(-t^\beta)  \backsimeq 1-\frac{t^\beta}{\Gamma(\beta +1)} \backsimeq e^{-\nicefrac{t^\beta}{\Gamma (\beta + 1)}},\, \text{for} \,0<t\leq 1 ,
\end{equation*}
and as a polynomial decay for $t\rightarrow \infty$,
\begin{equation*}
    E_{\beta}(-t^\beta)  \backsimeq \frac{\sin ({\beta \pi})}{\pi} \frac{\Gamma (\beta)}{t^\beta}, \text{for} \,t \geq 1,
\end{equation*}
the polynomial decay behaviour  of $E_{\beta}(-t^\beta)$ illustrates the need for the sharp condition of $\beta \in (0,\, \frac{1}{2}]$ in Theorem \ref{Thrm-1}. So the representation $G_B^{(\beta)}$  in equation \eqref{JNA-4} which is defined interms of $E_{\beta}(.)$ is crucial to our results.
\begin{theorem}\label{Thrm-3}
In Theorem \ref{Thrm-1}, suppose that $\beta\in (\frac{1}{2},\,1)$.
Suppose also that either  $d<\alpha/2\beta$ or  $\{\varphi_n\}_{n\geq1}$ are uniformly bounded by a constant $C(B)$, then there exist a strictly positive real number $\lambda_l$ such that for $\lambda<\lambda_l$  the quantity  $\sup_{t>0}\sup_{x\in B} \E|u_t(x)|^2$ is finite.
\end{theorem}

\begin{remark}
  The assumption of a uniform bound $C(B)$ on the eigenfunctions $\{\varphi_n\}_{n\geq1}$ look artificial, however there is no known bound for these eigenfunctions except the cases given in Lemma \ref{Lema-1} and the uniform bounds as in Example \ref{Ex-1} for the Brownian motion in higher dimensional rectangular boxes.
\end{remark}
From the  three results above in line with the interpretation given in Foondun \cite{Foondun2}, we note that whenever $\beta\in(0,\frac{1}{2}]$, the killed time changed $\alpha$-stable process $X_{E_t}^B$ defined in Section \ref{sect-prelim} reaches the boundary of $B$ slower which allows the non-linear term to grow exponentially for any $\lambda$ when $t$ becomes large. On the other hand, for $\beta \in (\frac{1}{2},\,1)$, this process does not have time to generate such growth unless $\lambda$ is large enough, since it proceeds quickly enough to the boundary.\\
Fractional parameters play an important role in many problems involving space-time fractional equations. However, in modeling problems, these fractional parameters are unknown a priori. As a result, continuity of the solutions with respect to these parameters is essential for modeling purposes.
The following continuity theorem of the unique solution $u_t^{\beta}(x)$ of equation \eqref{Dang} with respect to the parameter $\beta$ is Theorem 4.3(b) of Dang $et \,al.$ \cite{Dang}.

\begin{theorem}[Dang, Nane and Nguyen \cite{Dang}]\label{Dang}
Let $u_t^{(\gamma)}$ and $u_t^{(\beta)}$ denote the solution to the following equation for parameters $\gamma, \beta \in \left (0,\,1 \right)$  with $\gamma \rightarrow \beta$. The initial condition $u_0$ is the same for both equations.
\begin{equation}\label{Dang2}
\left \{\begin{aligned}
 \frac{\partial^\beta}{\partial t^\beta} u_t(x) &= \Delta u_t(x),\,x \in B,\,t>0, \\
u_t(x)&=0, \, x \in \partial D, \, t>0,\\
 u(0,\,x)&=f(x),\,x \in B.
\end{aligned}\right.
\end{equation}
Then, we have
\begin{equation*}
\lim_{\gamma \rightarrow \beta}||u_t^{(\gamma)}(\text{x})- u_t^{(\beta)}(\text{x})||_H^2=0.
\end{equation*}
where $H$ is the Hilbert space of all functions with the bounded norm induced by
$$||f||_H= \sqrt{\sum_{k=1}^{\infty} \mu_k^2|\langle f,\,\varphi_k \rangle|^2}.$$
\end{theorem}
\begin{remark}
  Meerschaert \textit{et al.} \cite{Meerschaert4} established the existence of a unique solution for equation \eqref{Dang2}.
\end{remark}
The following continuity theorem of the unique solution $u_t^{\beta}(x)$ of equation \eqref{JNA001} with respect to the parameter $\beta$ is Theorem 1.3 of Foondun \cite{Foondun2}.
\begin{theorem} [Foondun  \cite{Foondun2}]
Let $u_t^{(\beta)}$ and $u_t$ denote the solution to equation \eqref{JNA001} and the solution to equation \eqref{JNA001} for $\beta =1$ respectively. The initial condition $u_0$ is the same for both equations. Then, for any $p\geq 2$, we have
\begin{equation*}
\lim_{\beta\rightarrow 1} \sup_{\text{x} \in [0,\,L]}\E|u_t(x)-
u_t^{(\beta)}(x)|^p = 0.
\end{equation*}
\end{theorem}
In the next theorem, we show the continuity of the solution $u_t(x)$ of \eqref{JNA-9} with respect to the fractional parameter $\beta$.

\begin{theorem}\label{Thrm-4} Assume that $\{\varphi_n\}_{n\geq1}$ are uniformly bounded by a constant $C(B)$. For $d<\frac{1}{2} \min\{\frac{1}{\beta},\,\frac{1}{\gamma}\}\alpha$, let $u_t^{(\gamma)}$ and $u_t^{(\beta)}$ denote solutions to equation \eqref{JNA-9} for parameters $\beta, \gamma\in \left (\frac{1}{2},\,1 \right)$ respectively with $\gamma \rightarrow \beta$. The initial condition $u_0$ is the same for both equations. Then, for any $p\geq 2$, we have
\begin{equation*}
\lim_{\gamma \rightarrow \beta}\sup_{\text{x} \in B}\E|u_t^{(\gamma)}(\mathrm{x})- u_t^{(\beta)}(\mathrm{x})|^p=0.
\end{equation*}
\end{theorem}
Here we observe that Theorem \ref{Thrm-4} extends Theorem 1.3 in Foondun \cite{Foondun2} to fractional Laplacian case, and Theorem 4.3(b) of Dang \textit{et al.}  \cite{Dang} to stochastic case.

The other class of equations that we consider in this paper is equation with space colored noise stated as
\begin{equation}\label{JNA-8-b}
\left \{\begin{aligned}
&\partial^\beta_t u_t(\text{x})= -\left(-\Delta \right)^{\frac{\alpha}{2}} u_t(\text{x})+I_t^{1-\beta}[\lambda \sigma(u_t(\text{x}))\dot{F}(t,\,\text{x})]\;\; \text{for}\;\; \text{\text{x}} \in B\;\;\text{and}\; \;t>0\\
&u_t(\text{x})=0 \quad \text{for} \quad \text{x} \notin B\;\;\text{and}\; \;t>0,
\end{aligned}\right.
\end{equation}
and the initial condition $u_0:B \rightarrow \R_+$ is  a non-random measurable and bounded
function that has support with positive measure inside B. The operator $-\left(-\Delta \right)^{\frac{\alpha}{2}}$, where $0< \alpha \leq 2$, is the $L^2$-generator of a symmetric $\alpha$-stable process $X_t^B$ killed when exiting $B$. The function $\sigma:\R\rightarrow \R$ is a globally Lipschitz function satisfying $l_\sigma|\text{x}|\leq  |\sigma(\text{x})| \leq L_\sigma|\text{x}|$ where $l_\sigma$ and $L_\sigma$ are positive constants. The positive parameter $\lambda$ is called {\it the level of the noise}.\\
The noise $\dot{F}(t,\,x)$ is white in time and colored in space satisfying
\begin{equation*}
\Cov(\dot{F}(t,\,x), \dot{F}(s,\,y))=\delta_0(t-s)f(x,y),
\end{equation*}
where $0<f(x,\,y) \leq g(x-y)$ and $g$ is a locally integrable function on $\R^d$ with possible singularity at $0$ satisfying
\begin{equation} \label{Dalang}
    \int_{\R^d} \frac{\hat{g}(\xi)}{1+|\xi|^{\alpha}} \d \xi < \infty,
\end{equation}
where $\hat{g}$ denotes the Fourier transform of $g$.

We will need the following non degeneracy condition on the spatial correlation of the noise.
\begin{assumption}\label{lowerbound}
Assume there exists some positive number $K_{f}$ such that
\begin{equation*}
\inf_{x,\,y\in B}f(x,\,y)\geq K_{f}.
\end{equation*}
\end{assumption}
This assumption is very mild as it is shown by the following examples.

\begin{example} For the following list of examples Assumption \ref{Dalang} is satisfied.
\begin{itemize}
\item Riesz Kernel:
\begin{equation*}
    f(x,y)=\frac{1}{|x-y|^{\gamma}}\,\,\text{with}\,\,\gamma < d \wedge \alpha .
\end{equation*}
\item The Exponential-type kernel: $f(x,y)= exp[-(x \cdot y)]$.
\item The Ornstein-Uhlenbeck-type kernels: $f(x,\,y)=exp[-|x-y|^\delta]$ with $\delta \in (0,\,2]$.
\item Poisson Kernels:
\begin{equation*}
    f(x,y)=\left (\frac{1}{|x-y|^2+1}\right )^{\frac{d+1}{2}}.
\end{equation*}
\item Cauchy Kernels:
\begin{equation*}
    f(x,y)=\sum_{j=1}^d \left (\frac{1}{1+(x_j-y_j)^2 }\right ).
\end{equation*}
\end{itemize}

\end{example}

\noindent Following Walsh \cite{Walsh}, $u_t$ is a mild solution to \eqref{JNA-8-b} if
\begin{equation}\label{JNA-9-b}
u_t(x)=
(\mathcal{G}_B^{(\beta)} u_0)_t(\text{x})+ \lambda \int_B\int_0^t G_B^{(\beta)}(t-s,\,\text{x},\,\text{y})\sigma(u_s(\text{y}))F(\d s\,,\d \text{y}),
\end{equation}
where  $G_B^{(\beta)} (t,\text{x},\text{y})$ denotes the heat kernel of the space-time fractional diffusion equation with Dirichlet boundary conditions in \eqref{JNA-8-b}, and
\begin{equation*}
(\mathcal{G}_B^{(\beta)} u_0)_t(x):=\int_B G_B^{(\beta)}(t,\,\text{x},\,\text{y}) u_0(\text{y})\,\d \text{y}.
\end{equation*}

\begin{theorem}\label{Thrm-5}
Suppose that the Dalang condition \eqref{Dalang} holds. Let $u_t$ denote the unique solution to \eqref{JNA-9}.  Then no matter what $\lambda$ is, the second moment of $u_t$ cannot decay exponentially fast. In fact, if we further assume that $\beta\in (0,\,\frac{1}{2}]$ and Assumption \ref{lowerbound} holds, then as $t$ gets large, $\sup_{x\in B} \E|u_t(x)|^2 $ grows exponentially fast for any $\lambda.$
\end{theorem}

\begin{theorem}\label{Thrm-6} In Theorem \ref{Thrm-5}, if $\beta\in (\frac{1}{2},\,1)$ and Assumption \ref{lowerbound} hold, then there exist a strictly positive real number $\lambda_u$ such that for all $\lambda>\lambda_u$,  $\sup_{x\in B} \E|u_t(x)|^2 $ grows exponentially fast as time gets large.
\end{theorem}
\begin{theorem}\label{Thrm-7}
In Theorem \ref{Thrm-5}, suppose that $\beta\in (\frac{1}{2},\,1)$, $d<\frac{\alpha}{2\beta}$, and $\int_{B \times B} f(x,\,y)\, \d x \,\d y <\infty$. If  $\{\varphi_n\}_{n\geq1}$ are uniformly bounded by a constant $C(B)$, then there exist a strictly positive real number $\lambda_l$ such that for $\lambda<\lambda_l$  the quantity  $\sup_{t>0}\sup_{x\in B} \E|u_t(x)|^2$ is finite.
\end{theorem}

\begin{corollary} \label{Thrm-8}
In Theorem \ref{Thrm-7}, if $d=1$ and $\alpha = 2$ then the conclusion of theorem follows.
\end{corollary}
\begin{corollary} \label{Thrm-9}
In Theorem \ref{Thrm-7}, if $d=1$ and $f$ is Riesz Kernel function, then the conclusion of theorem follows.
\end{corollary}

\begin{theorem}\label{Thrm-10}  Assume that $\{\varphi_n\}_{n\geq1}$ are uniformly bounded by a constant $C(B)$, and ~\\$\int_{B \times B} f(x,\,y)\, \d x \,\d y <\infty$. For $d<\frac{1}{2} \min\{\frac{1}{\beta},\,\frac{1}{\gamma}\}\alpha$, let $u_t^{(\beta)}$ and $u_t^{(\gamma)}$ denote solutions to \eqref{JNA-9} for parameters $\beta, \, \gamma \in \left (\frac{1}{2},\,1 \right)$ with $\gamma \rightarrow \beta$. The initial condition $u_0$ is the same for both equations. Then, for any $p\geq 2$, we have
\begin{equation*}
\lim_{\gamma \rightarrow \beta}\sup_{\text{x} \in B}\E|u_t^{(\gamma)}(\text{x})- u_t^{(\beta)}(\text{x})|^p=0.
\end{equation*}
\end{theorem}
We now briefly give an outline of the paper. In this paper we employ similar methods as in Foondun \cite{Foondun2} with crucial changes to prove our main results. This method was also used in \cite{Foondun-Davar} and \cite{Foondun1} among others. Section \ref{sect-prelim} contains estimates and some preliminary results needed for the proof of main results. Section \ref{proof1} is devoted to the proof of Theorem \ref{Thrm-1}, Theorem \ref{Thrm-2}, and Theorem \ref{Thrm-3}  while the proofs of Theorem \ref{Thrm-4} is given in Section \ref{proof4}. Further, in Section \ref{proof3} we give the proof of Theorem \ref{Thrm-5}, Theorem \ref{Thrm-6}, Theorem \ref{Thrm-7}, and Theorem \ref{Thrm-10}.

\section{Preliminaries} \label{sect-prelim}
In this section we give some preliminary results, which are needed for the proofs of the main theorems. \\
Let $X_t$ denote a symmetric stable process of index $\alpha \in (0,2)$ in $\R^d$ and $B$ be a regular bounded open subset of $\R^d$. Let $X_t^B$ denote the symmetric stable process killed upon exiting $B$. The following Hilbert-Schmidt expansion (see Davies \cite{Davies} ) holds for the probability density function $p_B$ of $X_t^B$,
\begin{equation}\label{JNA-1}
p_B(t,\,\text{x},\,\text{y}):=\sum_{n=1}^\infty e^{-\mu_{n} t}\varphi_{n}(\text{x})\varphi_{n}(\text{y}),
\end{equation}
for all $\text{x}, \, \text{y} \in B,\, t>0,$ where $\{\varphi_{n}\}_{n\geq 0}$ is an orthonormal basis of $L^2(B)$, and $0<\mu_1<\mu_2\leq\mu_3\leq \cdots$ is a sequence of positive real numbers such that,  for every $n\geq1$, $P_{t}^B\varphi_{n} = e^{-\mu_{n} t}\varphi_{n}$.\\
From Theorem 2.3 of Blumenthal and Getoor \cite{Blumental} and the proof of Theorem 5.1 of Chen, Meerschaert, and Nane \cite{Chen}, we infer the following lemma:
\begin{lemma} \cite{Blumental, Chen} \label{Lema-1}
For any bounded open subset $B$ of $\R^d$, the system of eigenfunctions $\{\varphi_n\}_{n\geq1}$ and the corresponding eigenvalues $\{\mu_n\}_{n\geq1}$ satisfy the following:
\begin{enumerate}
\item $C_{1}n^{\frac{\alpha}{d}} \leq \mu_n \leq C_{2} n^{\frac{\alpha}{d}}$ for every $n \in \N$,
\item $|\varphi_n (\mathrm{x})|\leq C\mu_n^{\frac{d}{2\alpha}}$ for every $\mathrm{x} \in B$,
\end{enumerate}
where $C$, $C_1$, and $C_2$ are positive real numbers.
\end{lemma}
\noindent It is also known that
\begin{align*}
u_t(\mathrm{x})&:=\E^x[u_0(X_t^B)]\\
 &=\int_{B} p_B(t,\,\text{x},\,\text{y})\,u_0(y)\,\d y,
\end{align*}
solves the heat equation $\partial_t u_t(\mathrm{x})=-(-\Delta)^{\alpha/2}u_t(\text{x})$ defined on $B$ with Dirichlet boundary condition and initial condition $u_0$.

Let $D = \{D_r, r\geq 0\}$ denote a $\beta$-stable subordinator and $E_t$ to be the inverse of a stable subordinator of index $\beta\in (0,\,1)$. The process $X_{E_t}^B$ is just a time-changed of the killed $\alpha$-symmetric stable process $X^B_t$ and since $\beta\in (0,\,1)$, $X_{E_t}^B$ moves more slowly than $X_t^B.$ It is known that the density of the time changed process $X_{E_t}^B$ is given by the $G^{(\beta)}(t,x)$. By conditioning, we have
\begin{align}\label{JNA-2}
G^{(\beta)}(t, x)= \int_0^{\infty}p_B(s,x)f_t(s) ds,
\end{align}
where
\begin{align}\label{JNA-3}
f_t(s)=t\beta^{-1}s^{-1-1/\beta}g_\beta(ts^{-1/\beta}),
\end{align}
where $g_\beta(.)$ (Cf. Meerschaert and Straka \cite{Meerschaert5}) is the density function of $D_1$ and is infinitely differentiable on the entire real line, with $g_\beta(u)=0$ for $u \leq 0$; see Meerschaert and Scheffler \cite{Meerschaert3} for more information about the inverse stable subordinator $E_t$. \\
The function $v_t(\text{x}):=\E^x[u_0(X_{E_t}^B)]$ solves the space-time fractional equation
\begin{equation} \label{JNA-3a}
\left \{\begin{aligned}
&\partial^\beta_t v_t(\text{x})= -(-\Delta)^{\alpha/2} v_t(\text{x})\;\; \text{for}\;\; \text{x} \in B\;\;\text{and}\; \;t>0\\
&v_t(\text{x})=0 \quad \text{for}\;\; \text{x} \notin B,
\end{aligned}\right.
\end{equation}
with initial condition $u_0$ (Chen, Meerschaert, and Nane \cite{Chen}).  Thus, we get the following representation of $v_t (x)$
\begin{align} \label{JNA-3b}
v_t(\text{x})&:=\E^\text{x}[u_0(X_{E_t}^B)] \nonumber\\
&=\int_{B}\int_0^\infty \sum_{n=1}^\infty e^{-\mu_{n} s}\varphi_{n}(\text{x})\varphi_{n}(\text{y}) f_t(s)\,u_0(\text{y})\,\d s\d \text{y} \nonumber\\
&=\int_{B}\sum_{n=1}^\infty E_\beta(-\mu_{n} t^\beta)\varphi_{n}(\text{x})\varphi_{n}(\text{y})u_0(\text{y})\d \text{y} \nonumber\\
&=\int_{B}G_B^{(\beta)}(t,\,\text{x},\,\text{y})u_0(\text{y})\,\d \text{y},
\end{align}
this follows since the Laplace transform of $f_t(s)$ is $\hat{f}_t(\lambda)=E_\beta(-\lambda t^\beta)$ where $E_\beta(x)=\sum_{k=0}^\infty \frac{x^k}{\Gamma(1+\beta k)}$ is the
Mittag-Leffler function and have the following property,
\begin{equation}\label{JNA-6}
\frac{1}{1 + \Gamma(1-\beta)x}\leq E_{\beta}(-x)\leq \frac{1}{1+\Gamma(1+\beta)^{-1}x} \ \  \ \text{for}\ x>0,
\end{equation}
where $\Gamma(\cdot)$ is the gamma function. Thus, we get the following expansion,
\begin{equation}\label{JNA-4}
G_B^{(\beta)}(t,\,\text{x},\,\text{y}):= \sum_{n=1}^\infty E_\beta(-\mu_{n} t^\beta)\varphi_{n}(\text{x})\varphi_{n}(\text{y}).
\end{equation}
Using \eqref{JNA-3} and a change of variable from \eqref{JNA-4} we obtain
\begin{align}\label{JNA-5}
G_B^{(\beta)}(t,\,\text{x},\,\text{y})&=\int_0^\infty \sum_{n=1}^\infty e^{-\mu_{n} (t/u)^\beta}\varphi_{n}(\text{x})\varphi_{n}(\text{y}) g_\beta(u)\d u\\
&=\int_0^\infty p_B(\left(\frac{t}{u}\right)^\beta,\,\text{x},\,\text{y})g_\beta(u)\,\d u.
\end{align}
The following Lemma from Mijena and Nane \cite{Nane1} about the $L^2$-norm of the heat kernel which is a crucial ingredient for the existence and uniqueness of solution of equation \eqref{JNA-8} is also used in proving Theorem \ref{Thrm-4}.
\begin{lemma} \label{Lemma-2}
(Lemma 1 in \cite{Nane1}) Suppose that $d<2 \alpha$, then
\begin{equation}
    \int_{\R^d} [G_t^{(\beta)} (x)]^2 dx = C^*t^{\frac{-\beta d}{\alpha}},
\end{equation}
where the constant $C^*$ is given by
\begin{equation*}
    C^* = \frac{(\nu)^{-\frac{d}{\alpha}}2\pi^{\frac{d}{2}}}{\alpha \Gamma (\frac{d}{2})} \frac{1}{(2\pi)^d} \int_0^\infty z^{\frac{d}{\alpha-1}} (\E_\beta(-z))^2 dz.
\end{equation*}
\end{lemma}
The following three results are helpful in proving the finiteness result in Theorem \ref{Thrm-3}, and Theorem \ref{Thrm-7}.
\begin{lemma} \label{JNA-bd1}
For $p \geq 2$ and $\beta \in \left (\frac{1}{p}, \,1\right )$, we have
\begin{equation}
  \int_0^t E_\beta(-\mu_{n}(t-s)^\beta)^p \d s \lesssim \mu_n^{-\frac{1}{\beta}},
\end{equation}
for every $t>0$.
\end{lemma}

\begin{proof}
Using the Mittag-Leffler function bounds given in \eqref{JNA-6}, we get

\begin{align} \label{JNA-23}
\int_0^t E_\beta(-\mu_{n}(t-s)^\beta)^p \,\d s &\leq C\, \int_0^t \left (1+\Gamma (1+\beta)^{-1} \mu_n (t-s)^{\beta} \right)^{-p} \,\d s \nonumber\\
&\leq C\, \int_0^t \left (1 \wedge  \left (\mu_n (t-s)^{\beta} \right)^{-p} \right) \,\d s,
\end{align}
where

\begin{equation}\label{JNA-24}
1 \wedge \left (\mu_n (t-s)^{\beta} \right)^{-p} =
\begin{cases}
1 & \text{if }  \left (\mu_n (t-s)^{\beta} \right)^{-p} > 1,\\
\left (\mu_n (t-s)^{\beta} \right)^{-p} & \text{if }  \left (\mu_n (t-s)^{\beta} \right)^{-p} \leq 1.
\end{cases}
\end{equation}
Now
\begin{align} \label{JNA-25}
\left (\mu_n (t-s)^{\beta} \right)^{-p} \leq 1 &\implies \mu_n (t-s)^{\beta} \geq 1 \nonumber\\
&\implies  (t-s)^{\beta} \geq \frac{1}{\mu_n}  \nonumber\\
&\implies  t-s \geq \left (\frac{1}{\mu_n} \right) ^{\frac{1}{\beta}}  \nonumber\\
&\implies  t-\mu_n^{-\frac{1}{\beta}} \geq s.
\end{align}

So, from \eqref{JNA-24} using \eqref{JNA-25} we obtain
\begin{align} \label{JNA-26b}
\int_0^t E_\beta(-\mu_{n}(t-s)^\beta)^p \,\d s &\leq C\, \int_0^t \left (1 \wedge  \left (\mu_n (t-s)^{\beta} \right)^{-p} \right)  \,\d s\\
&= C  \,\left (\int_0^{t-\mu_n^{-\frac{1}{\beta}}} \left (\mu_n (t-s)^{\beta}\right)^{-p} \d s +\int_{t-\mu_n^{-\frac{1}{\beta}}}^t \d s \right) \nonumber\\
&= C   \,\left(\frac{1}{\mu_n^p}\int_0^{t-\mu_n^{-\frac{1}{\beta}}}    (t-s)^{-p\beta} \d s \nonumber+\mu_n^{-\frac{1}{\beta}} \right) \nonumber\\
&= C  \,\left (\frac{1}{\mu_n^p(1-p\beta)}\left(t^{-p\beta+1}-\left (\frac{1}{\mu_n} \right) ^{-p+\frac{1}{\beta}}\right) + \mu_n^{-\frac{1}{\beta}} \right) \nonumber\\
&\leq C \,\left(\frac{\mu_n^{p-\frac{1}{\beta}}}{\mu_n^p(p\beta -1)} + \mu_n^{-\frac{1}{\beta}} \right )  \nonumber\\
 &\leq C \,\left (\frac{1}{p\beta-1}+1 \right) \mu_n^{-\frac{1}{\beta}} \,\nonumber\\
&\lesssim \mu_{n}^{-\frac{1}{\beta}},
\end{align}
since $\frac{t^{-p\beta+1}}{1-p\beta}<0$.

\end{proof}

\begin{lemma} \label{JNA-45}
For $p \geq 2$, $\beta \in \left (\frac{1}{p}, \,1\right )$ and $d<\frac{\alpha}{\beta}$, we have
\begin{equation}
  \int_0^t\sum_{n=1}^\infty E_\beta(-\mu_{n}(t-s)^\beta)^p \d s < \infty,
\end{equation}
for every $t>0$.
\end{lemma}
\begin{proof} Using Lemma \ref{Lema-1}(a) and Lemma \ref{JNA-bd1}, we obtain the following:
\begin{align} \label{JNA-23-b}
\int_0^t\sum_{n=1}^\infty E_\beta(-\mu_{n}(t-s)^\beta)^p \d s &= \sum_{n=1}^\infty \int_0^t E_\beta(-\mu_{n}(t-s)^\beta)^p \d s \nonumber\\
&\lesssim \sum_{n=1}^\infty \mu_n^{-\frac{1}{\beta}} \, \nonumber\\
&\lesssim \sum_{n=1}^\infty n^{-\frac{\alpha}{\beta d}} < \infty,
\end{align}
since $\frac{\alpha}{\beta d}>1$.
\end{proof}

\begin{lemma} \label{JNA-45b}
For $p \geq 2$,  $\beta \in \left (\frac{1}{p}, \,1\right ),$ if $d<\frac{\alpha}{2\beta}$ then we have
\begin{equation}
  \int_0^t\sum_{n=1}^\infty E_\beta(-\mu_{n}(t-s)^\beta)^p \varphi_n^2(\text{x}) \,\d s < \infty,
\end{equation}
for every $t>0$.
\end{lemma}
\begin{proof} We prove the result using by Lemma \ref{Lema-1}(b)  and Lemma \ref{JNA-bd1} as follows
\begin{align} \label{JNA-23b}
\int_0^t\sum_{n=1}^\infty E_\beta(-\mu_{n}(t-s)^\beta)^p \varphi_n^2(\text{x}) \,\d s &= \sum_{n=1}^\infty \int_0^t E_\beta(-\mu_{n}(t-s)^\beta)^p \varphi_n^2(\text{x}) \,\d s \nonumber\\
&\leq C\,\sum_{n=1}^\infty \mu_n^{\frac{d}{\alpha}} \mu_n^{-\frac{1}{\beta}}  \nonumber\\
&\lesssim \,\sum_{n=1}^\infty n^{\frac{\alpha}{d}(\frac{d}{\alpha}-\frac{1}{\beta})} \, \nonumber\\
&\lesssim \,\sum_{n=1}^\infty n^{1-\frac{\alpha}{\beta d}} < \infty,
\end{align}
since $\frac{\alpha}{\beta d}>2$.
\end{proof}

The following continuity Lemma is crucial to justify the conclusion of Theorem \ref{Thrm-4}.
\begin{lemma}\label{JNA-27} Assume that $\{\varphi_n\}_{n\geq1}$ are uniformly bounded by a constant $C(B)$, depending on the geometrical characteristics of the domain $B$, i.e.,
\begin{equation*}
    C(B) = \sup_{n\geq 1,\text{x} \in B} \varphi_n(\text{x}).
\end{equation*}
Fix $t>0$, then for $d < \alpha$ we have
\begin{equation*}
\lim_{\gamma \rightarrow \beta} \sup_{x \in B} \left |(\mathcal{G}_B^{(\gamma)} u_0)_t(x)-(\mathcal{G}_B^{(\beta)} u_0)_t(x) \right |=0,
\end{equation*}
where
\begin{equation*}
(\mathcal{G}_B^{(\gamma)} u_0)_t(x):=\int_B G_B^{(\gamma)}(t,\,\text{x},\,\text{y}) u_0(\text{y})\,\d \text{y}.
\end{equation*}
\end{lemma}
\begin{proof}
Since the initial datum $u_0$ and $G_B^{(\gamma)} (t,\,x,\,y)$ are given to be bounded from above, it suffices to show that
\begin{equation*}
\lim_{\gamma \rightarrow \beta} \sup_{x \in B} \left |G_B^{(\gamma)}(t,\,x,\,y)-G_B^{(\beta)} (t,\,x,\,y)\right |=0.
\end{equation*}
Now using the Laplace transform of the Mittag-Leffler function, for any $\theta >0$ we have
\begin{align*}
\lim_{\gamma \rightarrow \beta}\int_0^\infty e^{-\theta t}E_\gamma(-\mu_n t^\gamma) \,\d t &= \lim_{\gamma \rightarrow \beta}\frac{\theta^{\gamma-1}}{\theta^\gamma+\mu_n}=\int_0^\infty e^{-\theta t}E_\beta(-\mu_n t^\beta) \,\d t.
\end{align*}
This implies that
\begin{align}\label{JNA-28}
\lim_{\gamma \rightarrow \beta} E_\gamma(-\mu_n t^\gamma)= E_\beta(-\mu_n t^\beta).
\end{align}
Now using the expansions for the heat kernel, we have
\begin{align} \label{JNA-29}
|G_B^{(\beta)}(t,\,x,\,y)-G_B^{(\gamma)}(t,\,x,\,y)|&\leq \sum_{n=1}^\infty \left [ |E_\beta(-\mu_n t^\beta)- E_\gamma (-\mu_nt^\gamma)||\varphi_n (\text{x})||\varphi_n (\text{y})|\right].
\end{align}
Taking limit as $\gamma \rightarrow \beta$ on both sides of \eqref{JNA-29} we obtain the required result, since each terms in the summation can be bounded by a quantity independent of $\gamma$ and $\beta$. This is due to the uniform boundedness of the eigenfunctions $\{\varphi_n\}_{n\geq1}$ and the bounds on the Mittag-Leffler function together with Lemma \ref{Lema-1}(a), and summable  as shown below.\\ For any  $t>0$, then we obtain
\begin{align} \label{JNA-29b}
\sum_{n=1}^\infty \bigg[ |E_\beta(-\mu_n t^\beta)  -E_\gamma (-\mu_n t^\gamma)|&|\varphi_n (\text{x})||\varphi_n (\text{y})| \bigg] \leq [C(B)]^2 \sum_{n=1}^\infty \left | E_\beta(-\mu_n t^\beta)- E_\gamma (-\mu_n t^\gamma) \right| \nonumber \\
&\leq [C(B)]^2\sum_{n=1}^\infty \left [ \frac{1}{1 + \Gamma(1+\beta)^{-1}\mu_n t^\beta} + \frac{1}{1 + \Gamma(1+\gamma)^{-1}\mu_n t^\gamma}\right] \nonumber \\
&\leq [C(B)]^2\sum_{n=1}^\infty \frac{1}{\mu_n} \left (\frac{1}{t^\beta} + \frac{1}{t^\gamma} \right) \nonumber \\
&\leq 2[C(B)]^2 \text{max}(1, 1/t)\sum_{n=1}^\infty n^{-\frac{\alpha}{d}}< \infty,
\end{align}
where $C$ is a constant positive real number.
\end{proof}
The following lemma is useful in the proof of Lemma \ref{nane-tuan}(b) and Proposition \ref{nane-tuan-2022-b}.
\begin{lemma} \label{Integral-bound} If $\beta\in (\beta_0, \beta_1)$ for some $\beta_0, \beta_1 \in (0,1)$ and $\frac{d}{\alpha}<\lambda<\frac{1}{2\beta}$, then using the inequality \\$e^{-z} \le  \mathscr C_\mu z^{-\mu} $, we bound the following integral as follows:
\begin{equation}
    \int_0^\infty e^{-2\theta s/p} \frac{1}{s^{2\lambda \beta}} \d s  \lesssim  p^{1-2\lambda \beta_0}\left [ \frac{\mathscr C_{\mu_0}^2 \theta^{-2\mu_0}}{1-2\lambda  -2\mu_0}  +  \frac{\mathscr C_{\mu_1}^2 \theta^{-2\mu_1}}{2\lambda \beta_0 +2\mu_1-1} \right],
\end{equation}
where $0< \mu_0 < \min\Big(\frac{1}{2}- \lambda ,~~\frac{1-2\beta_1 \lambda}{2} \Big)$ and $\mu_1> \frac{1- 2 \beta_0 \lambda}{2}$.
\end{lemma}
\begin{proof} Consider the integral
\begin{align}
    \int_0^\infty e^{-2\theta s/p} \frac{1}{s^{2\lambda \beta}} \d s &= p^{1-2\lambda \beta}\int_0^\infty  \frac{e^{-2\theta r}}{r^{2\lambda \beta}}  \d r \nonumber\\
     &=  p^{1-2\lambda \beta}\left [\int_0^1  \frac{e^{-2\theta r}}{r^{2\lambda \beta}}  \d r +\int_1^\infty  \frac{e^{-2\theta r}}{r^{2\lambda \beta}}  \d r\right]\nonumber\\
     &\lesssim  p^{1-2\lambda \beta}\left [\mathscr C_{\mu_0}^2 \theta^{-2\mu_0}\int_0^1  \frac{1}{r^{2\lambda \beta+2\mu_0}}  \d r +\mathscr C_{\mu_1}^2 \theta^{-2\mu_1}\int_1^\infty  \frac{1}{r^{2\lambda \beta+2\mu_1}}  \d r\right]\nonumber\\
    &\lesssim  p^{1-2\lambda \beta_0}\left [ \frac{\mathscr C_{\mu_0}^2 \theta^{-2\mu_0}}{1-2\lambda  -2\mu_0}  +  \frac{\mathscr C_{\mu_1}^2 \theta^{-2\mu_1}}{2\lambda \beta_0 +2\mu_1-1} \right],
\end{align}
where $\beta\in (\beta_0, \beta_1)$ for some $\beta_0, \beta_1 \in (0,1)$,   $\mu_0 < \min\Big( \frac{1}{2}- \lambda,~~\frac{1-2\beta_1 \lambda}{2}\Big)<\min\Big( \frac{1}{2}- \lambda,~~\frac{1-2\beta \lambda}{2}\Big)$ and $\mu_1> \frac{1- 2 \beta_0 \lambda}{2}$, which implies that the proper integral $\int_0^1  \frac{1}{r^{2\beta \lambda + 2\mu_0}} dr$ and $\int_1^\infty   \frac{1}{r^{2\beta \lambda + 2\mu_1}} dr$ are convergent.
\end{proof}
\begin{lemma}[Nane and Tuan, 2022 \cite{nane-tuan-2022}] \label{nane-tuan} Assume that $\{\varphi_n\}_{n\geq1}$ are uniformly bounded by a constant $C(B)$, depending on the geometrical characteristics of the domain $B$, i.e.,
\begin{equation*}
    C(B) = \sup_{n\geq 1,\text{x} \in B} \varphi_n(\text{x}).
\end{equation*}
	a) Let $\beta\in (\frac{1}{2},1)$, and   $\frac{d}{\alpha}<\lambda<\frac{1}{2\beta}$. Then there exists a constant $C=C(\lambda)$ independent of $\beta$ such that
	\begin{align} \label{G1}
	|G_B^{(\beta)}(t,\,x,\,y)| \leq  C {t^{-\beta \lambda}}.
	\end{align}	
	b)  Let $\beta\in (\beta_0, \beta_1)$ for some $\beta_0, \beta_1 \in (0,1)$. Let also $0< \mu_0 < \min\Big(\frac{1}{2}- \lambda ,~~\frac{1-2\beta_1 \lambda}{2} \Big)$ and $\mu_1> \frac{1- 2 \beta_0 \lambda}{2}$. Then for $d < \frac{\alpha}{2\beta}$, there exists positive constant $C$ which depends on $\mu_0, \mu_1, p$ and independent of $\beta$ such that
\begin{align} \label{a12}
    \bigg[\int_0^\infty\int_B & e^{-\frac{2\theta s}{p}} |G_B^{(\beta)}(s,\,x,\,y)|^2 \d y \,\d s \,\bigg]^\frac{p}{2} \lesssim C   p^{\frac{p(1-2\lambda \beta_0)}{2}}\left ( \theta^{-2\mu_0}  +  \theta^{-2\mu_1} \right)^{p/2}.
\end{align}
\end{lemma}
\begin{proof}
	(a) Using the fact that $E_{\beta} (-\mu_n t^\beta )\le \frac{C}{(1+\mu_n t^\beta)^\lambda}$ for any $\frac{d}{\alpha}<\lambda<\frac{1}{2\beta} $ and $C$ is independent of $\beta$,  we find that
	\begin{align}
	    |G_B^{(\beta)}(t,\,x,\,y)| &=\bigg|\sum_{n=1}^\infty E_\beta(-\mu_ns^\beta)\varphi_n(x)\varphi_n(y)\bigg|\lesssim C \sum_{n=1}^\infty |E_\beta(-\mu_ns^\beta)| \nonumber\\
	    &\lesssim C \sum_{n=1}^\infty \frac{1}{\mu_n^{\lambda}} \frac{1}{t^{\beta \lambda}}\lesssim C  {t^{-\beta \lambda}},
	\end{align}
	where we used the eigenfunction expansion of $G_B^{(\beta)}(t,\,x,\,y)$  and  that the series $\sum_{n=1}^\infty \frac{1}{\mu_n^{\lambda}}\leq \sum_{n=1}^\infty n^{-\frac{\alpha \lambda}{d}}$ is convergent since $\lambda>\frac{d}{\alpha}$.\\
	(b) Using part (a) of this lemma  together with  Lemma \ref{Integral-bound}, we obtain
\begin{align} \label{G4}
    \int_0^\infty  \int_Be^{-2\theta s/p}\big|G_B^{(\beta)}(s,\,x,\,y) \big|^2\d y \,\d s &\lesssim  \int_0^\infty e^{-2\theta s/p} \frac{1}{s^{2\lambda \beta}} \d s \nonumber\\
   &\lesssim  p^{1-2\lambda \beta_0}\left [ \frac{\mathscr C_{\mu_0}^2 \theta^{-2\mu_0}}{1-2\lambda  -2\mu_0}  +  \frac{\mathscr C_{\mu_1}^2 \theta^{-2\mu_1}}{2\lambda \beta_0 +2\mu_1-1} \right],
\end{align}
where $\beta\in (\beta_0, \beta_1)$ for some $\beta_0, \beta_1 \in (0,1)$,   $\mu_0 < \min\Big( \frac{1}{2}- \lambda,~~\frac{1-2\beta_1 \lambda}{2}\Big)<\min\Big( \frac{1}{2}- \lambda,~~\frac{1-2\beta \lambda}{2}\Big)$ and $\mu_1> \frac{1- 2 \beta_0 \lambda}{2}$, which implies that the proper integral $\int_0^1  \frac{1}{r^{2\beta \lambda + 2\mu_0}} dr$ and $\int_1^\infty   \frac{1}{r^{2\beta \lambda + 2\mu_1}} dr$ are convergent. It  follows from \eqref{G4} that for any $p \ge 0$
	\begin{align*}
\bigg[\int_0^\infty\int_B e^{-2\theta s/p}|& G_B^{(\beta)}(s,\,x,\,y)|^2\,\d s\,\d y \bigg]^{p/2} \lesssim p^{\frac{p(1-2\lambda \beta_0)}{2}}\left [ \frac{\mathscr C_{\mu_0}^2 \theta^{-2\mu_0}}{1-2\lambda -2\mu_0}  +  \frac{\mathscr C_{\mu_1}^2 \theta^{-2\mu_1}}{2\lambda \beta_0 +2\mu_1-1} \right] ^{p/2},
\end{align*}
which allows us to deduce \eqref{a12}.
That is, each terms in the summation can be bounded by a quantity independent of $\beta$.

\end{proof}

Next we give an example where the sufficient condition of Lemma \ref{JNA-27} is satisfied.
\begin{example} \label{Ex-1} Let $X_t$ denote a Brownian motion in $\R^2$ and $X_t^B$ denote the Brownian motion killed upon exiting the rectangular domain $B:=[0,\,L_{1}]\times[0,\,L_{2}]$. The eigenvalues of the Dirichlet Laplacian are $\mu_{m,n}=(\frac{m\pi}{L_{1}})^2 + (\frac{n\pi}{L_{2}})^2$  and $\varphi_{m,n}(x, y):=\varphi_{m}(x)\varphi_{n}( y)=\frac{2}{\sqrt{L_{1}L_{2}}}\sin(\frac{m\pi x}{L_{1}})\sin(\frac{n\pi y}{L_{2}})$ are the corresponding eigenfunction so that $|\varphi_{m,n}(x. y)|\leq \frac{2}{\sqrt{L_{1}L_{2}}}$ for all $(x,y)\in B$.  A similar result is valid for the Brownian motion in higher dimensional rectangular boxes.
\end{example}

\section{Proof of Theorem \ref{Thrm-1},Theorem \ref{Thrm-2}, and Theorem \ref{Thrm-3}} \label{proof1}
In this section we give the proof of first three main theorems of the paper. We denote
\begin{equation}\label{JNA-7}
\Lambda(\theta):=\int_0^\infty e^{-\theta t}E_\beta(-\mu_{1}t^\beta)^2\,\d t,
\end{equation}
for $\theta>0$. Here it is crucial to note that we can use the bounds in \eqref{JNA-6} to observe that $\Lambda(\theta)$ tends to infinity as $\theta$ approaches to zero if and only if  $2\beta\leq 1$.
\begin{proof}[{\bf Proof of Theorem \ref{Thrm-1}}]
Let us first establish the second statement of the theorem. From the mild formulation given in \eqref{JNA-9} and using Stochastic Fubini theorem we set
\begin{align} \label{JNA-10b}
\langle u_t, \varphi_{1}\rangle &=\int_B u_t(x) \varphi_{1}(\text{x})\,\d \text{x} \nonumber\\
&=E_\beta(-\mu_{1} t^\beta)\langle u_0, \varphi_{1}\rangle\, \nonumber\\&~~~~~~~~+\lambda \int_B   \int_0^t E_\beta(-\mu_{1} (t-s)^\beta)\varphi_{1}(\text{y})\sigma(v_s(\text{y}))W(\d s\,,\d \text{y}).
\end{align}
Taking the second moment of \eqref{JNA-10b} and using the Ito-Walsh isometry, we get
\begin{align} \label{JNA-11b}
\E\langle u_t, \varphi_{1}\rangle^2&=E_\beta(-\lambda_1 t^\beta)^2 \langle u_0, \varphi_{1}\rangle^2\nonumber\\ &~~~~~~+\lambda^2\int_B\int_0^t E_\beta(-\lambda_1(t-s)^\beta)^2\varphi_{1}^2(\text{y})\E|\sigma(u_s(\text{y}))|^2\d s\d \text{y}.
\end{align}
Now using \eqref{JNA-11b} and the assumption on $\sigma$ gives
\begin{align}\label{JNA-12a}
\int_0^\infty e^{-\theta t}\E\langle u_t, \varphi_{1}\rangle^2\,\d t &= \Lambda (\theta) \,\langle u_0, \varphi_{1}\rangle^2 +\lambda^2 \Lambda (\theta) \,\int_0^\infty e^{-\theta t}\int_B \E|\sigma(u_s(\text{y}))|^2\,\varphi_{1}^2(\text{y}) \d \text{y} \,\d t. \nonumber \\
&\gtrsim  \Lambda(\theta) \langle u_0, \varphi_{1}\rangle^2+\lambda^2l_\sigma^2\Lambda(\theta)\int_0^\infty e^{-\theta t}\E\langle u_t, \varphi_{1}\rangle^2\,\d t.
\end{align}
If $\beta\in(0,\,\frac{1}{2}]$, since $\Lambda(\theta)$ tends to infinity as $\theta$ goes to zero, we can choose $\theta$ small enough and from \eqref{JNA-12a} we obtain
\begin{align}\label{JNA-13b}
\int_0^\infty e^{-\theta t}\E\langle u_t, \varphi_{1}\rangle^2\,\d t &\gtrsim  \Lambda(\theta) \langle u_0, \varphi_{1}\rangle^2+2\int_0^\infty e^{-\theta t}\E\langle u_t, \varphi_{1}\rangle^2\,\d t.
\end{align}
This implies that that for sufficiently small $\theta$, we have
\begin{equation*}
\int_0^\infty e^{-\theta t}\E\langle u_t, \varphi_{1}\rangle^2\,\d t=\infty,
\end{equation*}
which implies that for sufficiently large $t$, $\E\langle u_t, \varphi_{1}\rangle^2$ grows exponentially. Now using Cauchy-Schwarz inequality we get the following,


\begin{align*}
\E\langle u_t, \varphi_{1}\rangle^2&\leq \E \left ( \left(\int_B |u_t(x)|^2 \,\d x \right) \left (\int_B |\varphi_1(x)|^2\,\d x \right)\right)\\
&\lesssim  \sup_{\text{x}\in B}\E|u_t(\text{x})|^2.
\end{align*}
We can thus conclude that $\sup_{\text{x}\in B}\E|u_t(\text{x})|^2$ too grows exponentially fast for all values of $\lambda$. \\
Note that the first part of the conclusion of the theorem merely follows from the fact that the second term of \eqref{JNA-13b} is positive and that the first term cannot have exponential decay.
\end{proof}


\begin{proof}[\bf Proof of  Theorem \ref{Thrm-2}]
Following the lines for the proof of Theorem \ref{Thrm-1}, we get the inequality \eqref{JNA-12a};
\begin{align}\label{JNA-13}
\int_0^\infty e^{-\theta t}\E\langle u_t, \varphi_{1}\rangle^2\,\d t&\gtrsim  \Lambda(\theta) \langle u_0, \varphi_{1}\rangle^2+\lambda^2l_\sigma^2\Lambda(\theta)\int_0^\infty e^{-\theta t}\E\langle u_t,\,\varphi_{1}\rangle^2\,\d t.
\end{align}
Since $\beta\in (\frac{1}{2},\,1)$, the function $\Lambda(\theta)$ is bounded. Thus, for any fixed $\theta>0$ we can find sufficiently large $\lambda_u$ so that for all $\lambda\geq \lambda_u$, the above inequality \eqref{JNA-13} yields
\begin{align} \label{JNA-14}
\int_0^\infty e^{-\theta t}\E\langle u_t, \varphi_{1}\rangle^2\,\d t&\gtrsim  \Lambda(\theta) \langle u_0, \varphi_{1}\rangle^2+\frac{3}{2}\int_0^\infty e^{-\theta t}\E\langle u_t, \varphi_{1}\rangle^2\,\d t.
\end{align}
The result of the theorem is obtained by following the arguments of the remaining part of the proof of Theorem \ref{Thrm-1}.
\end{proof}

\begin{proof}[{\bf Proof of  Theorem \ref{Thrm-3}}]
We establish the required result using Ito-Walsh isometry and the global Lipschitz assumption on $\sigma$. The second moment of the mild formulation given by \eqref{JNA-9} becomes:
\begin{align} \label{JNA-15}
\E|u_t(x)|^2&=
|(\mathcal{G}_B^{(\beta)} u)_t(\text{x})|^2+ \lambda^2 \int_B\int_0^t G_B^{(\beta)}(t-s,\,\text{x},\,\text{y})^2\E|\sigma(u_s(\text{y}))|^2\d s\,\d \text{y} \nonumber\\
&\leq |(\mathcal{G}_B^{(\beta)} u)_t(\text{x})|^2+ \lambda^2 L_\sigma^2\int_B\int_0^t G_B^{(\beta)}(t-s,\,\text{x},\,\text{y})^2\E|u_s(\text{y})|^2\d s\,\d \text{y} \nonumber\\
&:=J_1+J_2.
\end{align}
We observe $J_1$ is bounded by the square of the same constant as initial condition, since \\$\int_B G_B^{(\beta)}(t,\,x)\,\d x \leq \int_{\R^d} G^{(\beta)}(t,\,x)\,\d x =1$ and the initial datum is assumed to be bounded above by a constant. It remains to bound $J_2$. Since $\{\varphi_{n}\}_{{n} \geq 1}$ is an orthonormal sequence, for each fixed $t > 0$ we obtain
\begin{align} \label{JNA-16}
J_2&:=\lambda^2 L_\sigma^2\int_B\int_0^t G_B^{(\beta)}(t-s,\,\text{x},\,\text{y})^2\E|u_s(\text{y})|^2\d s\,\d \text{y} \nonumber\\
&\leq a_t\lambda^2 L_\sigma^2 \int_0^t \sum_{n=1}^\infty E_\beta(-\mu_{n}(t-s)^\beta)^2 \varphi_{n}^2(\text{x})\, \d s <\infty,
\end{align}
by Lemma \ref{JNA-45b} in the case of $d<\alpha/2\beta$, where $a_t=\sup_{0<s<t}\sup_{\text{x}\in B} \E|u_s(\text{x})|^2$. In the case the eigenfunctions $\{\varphi_n\}_{n\geq1}$ are uniformly bounded by a constant $C(B)$, it follows from the above that
\begin{align} \label{JNA-16b}
   J_2 : &\leq a_t\lambda^2 L_\sigma^2 \int_0^t \sum_{n=1}^\infty E_\beta(-\mu_{n}(t-s)^\beta)^2 \varphi_{n}^2(\text{x})\, \d s \nonumber \\
   & \leq a_t\lambda^2 L_\sigma^2 \left (C(B)\right)^2 \int_0^t \sum_{n=1}^\infty E_\beta(-\mu_{n}(t-s)^\beta)^2 \d s \leq a_t\lambda^2 L_\sigma^2 \left (C(B)\right)^2 C(ML),
\end{align}
 where $a_t=\sup_{0<s<t}\sup_{\text{x}\in B} \E|u_t(\text{x})|^2$ and  $C(ML)=\displaystyle\int_0^t \sum_{n=1}^\infty E_\beta(-\mu_{n}(t-s)^\beta)^2 \d s$.  By Lemma \ref{JNA-45} $C(ML)<\infty$ is a finite constant independent of $t$. That is, since $\beta\in (\frac{1}{2}, 1)$,  we can choose $\lambda_l$ sufficiently small that for all $\lambda\leq \lambda_l$, the above estimates shows
\begin{align*}
\sup_{0<s<t}\sup_{\text{x}\in B} \E|u_t(\text{x})|^2\lesssim 1+\frac{1}{2}\sup_{0<s<t}\sup_{\text{x}\in B} \E|u_t(\text{x})|^2.
\end{align*}
This shows that $\sup_{0<t<\infty}\sup_{\text{x}\in B} \E|u_t(\text{x})|^2$ is finite and hence the conclusion of the theorem follows.
\end{proof}

\section{Proof of Theorem \ref{Thrm-4}}\label{proof4}
In this section we give the proof of Theorem \ref{Thrm-4} and the following proposition is used in its proof.  We can use the bounds of $G_B^{(\beta)}(t,\,x,\,y)$ to show the following proposition holds:
\begin{proposition} For $d < \frac{\alpha }{\beta}$, let $u_t^{(\beta)}$ be a solution to equation \eqref{JNA-9} for parameters $\beta \in \left (\frac{1}{2},\,1 \right)$.
Then for some $\theta$, the supremum on the $p^{th}$ moment of the solution $u_s^{(\beta)}(x)$ is given by
\begin{equation} \label{JNA-42}
\sup_{t>0,\,x\in B}e^{-\theta t}\E|u_s^{(\beta)}(x)|^p,
\end{equation}
is bounded above by a constant independent of $\beta$.
\end{proposition}
\begin{proof} We use the mild formulation in equation \eqref{JNA-9} and follow similar computations to those used in the above proof. Now consider
\begin{equation}\label{JNA-46}
u_t^{(\beta)}(x)=
(\mathcal{G}_B^{(\beta)} u_0)_t(\text{x})+ \lambda \int_B\int_0^t G_B^{(\beta)}(t-s,\,\text{x},\,\text{y})\sigma(u_s^{(\beta)}(\text{y}))W(\d s\,,\d \text{y}),
\end{equation}
Now applying the inequality $(a+b)^p \leq 2^p(a^p\,+b^p)$ for any $a,b \geq 0$ and using Burkholder-Davis-Gundy inequality, the $p^{th}$ moment of \eqref{JNA-46} becomes
\begin{equation}\label{JNA-47}
\E|u_t^{(\beta)}(x)|^p\leq
2^p|(\mathcal{G}_B^{(\beta)} u_0)_t(\text{x})|^p+ (2\lambda)^p \left [\int_B\int_0^t |G_B^{(\beta)}(t-s,\,\text{x},\,\text{y})|^2\left [\E|\sigma(u_s^{(\beta)}(\text{y}))|^p \right]^{\frac{2}{p}}\d s\,\d \text{y} \right]^{\frac{p}{2}},
\end{equation}
\begin{align}\label{JNA-48}
e^{-\theta t}\E|u_t^{(\beta)}(x)|^p &\lesssim
e^{-\theta t}|(\mathcal{G}_B^{(\beta)} u_0)_t(\text{x})|^p+ e^{-\theta t} \left [\int_B\int_0^t |G_B^{(\beta)}(t-s,\,\text{x},\,\text{y})|^2\left [\E|\sigma(u_s^{(\beta)}(\text{y}))|^p \right]^{\frac{2}{p}} \d s\,\d \text{y} \right]^{\frac{p}{2}} \nonumber\\
&:=J_1 + J_2.
\end{align}
Using Lemma \ref{Lemma-2} and since $\sigma$ is globally Lipschitz inequality, the second term becomes
\begin{align} \label{JNA-49}
J_2 &:=e^{-\theta t} \left [\int_B\int_0^t |G_B^{(\beta)}(t-s,\,\text{x},\,\text{y})|^2\left [\E|\sigma(u_s^{(\beta)}(\text{y}))|^p \right]^{\frac{2}{p}} \d s\,\d \text{y}\right]^{\frac{p}{2}} \nonumber\\
&\leq  L_{\sigma}^p \left [\int_B\int_0^tG_B^{(\beta)}(t-s,\,\text{x},\,\text{y})^2 e^{\frac{-2\theta t}{p}}[\E|u_s^{(\beta)}(\text{y})|^p]^{2/p}\d s\,\d y \right]^{p/2} \nonumber\\
&\leq  L_{\sigma}^p b_t(\theta) \left [\int_B\int_0^tG_B^{(\beta)}(t-s,\,\text{x},\,\text{y})^2 e^{\frac{-2\theta (t-s)}{p}} \d s\,\d y \right]^{p/2}\nonumber\\
&\lesssim  L_{\sigma}^p b_t(\theta) \left[\Gamma\left (1-\frac{\beta d}{\alpha}\right)\theta ^{\frac{\beta d}{\alpha}-1}\right]^{p/2},
\end{align}
where $b_t(\theta)=\sup_{0<s<t}\sup_{\text{z}\in B}e^{-\theta s}\E|u_s^{(\beta)}(\text{z})|^p$. Now we fix $\theta > 0$ sufficiently large so that
\begin{align} \label{JNA-50}
J_2 \lesssim \frac{1}{2} b_t(\theta).
\end{align}
Thus, we have
\begin{align}\label{JNA-51}
e^{-\theta t}\E|u_t^{(\beta)}(x)|^p &\lesssim 1+\frac{1}{2} b_t(\theta).
\end{align}
Taking the supremum of the left side of \eqref{JNA-51} the result follows.
\end{proof}
We are now ready to prove  Theorem \ref{Thrm-4}.
\begin{proof}[\bf Proof of Theorem \ref{Thrm-4}]
From the mild formulation of the solutions, we have
\begin{align}\label{JNA-17}
u_t^{(\gamma)}(x)-u_t^{(\beta)}(x)&=(\mathcal{G}_B^{(\gamma)} u)_t(x)-(\mathcal{G}_B^{(\beta)}u)_t(x)+\lambda \int_B\int_0^tG_B^{(\gamma)}(t-s,\,x,\,y)\sigma(u_s^{(\gamma)}(y))W(\d s\,\d y) \nonumber\\
&-\lambda \int_B\int_0^tG_B^{(\beta)}(t-s,\,x,\,y)\sigma(u_s^{(\beta)}(y))W(\d s\,\d y)\nonumber\\
&=(\mathcal{G}_B^{(\gamma)} u)_t(x)-(\mathcal{G}_B^{(\beta)}u)_t(x) \nonumber\\
&+\lambda\int_B\int_0^t[G_B^{(\gamma)}(t-s,\,x,\,y)-G_B^{(\beta)}(t-s,\,x,\,y)]\sigma(u_s^{(\gamma)}(y))W(\d s\,\d y)\\
&+\lambda\int_B\int_0^tG_B^{(\beta)}(t-s,\,x,\,y)[\sigma(u_s^{(\gamma)}(y))-\sigma(u_s^{(\beta)}(y))]W(\d s\,\d y).\nonumber
\end{align}
Now applying the inequality $(a+b+c)^p \leq 3^p(a^p\,+b^p\,+c^p)$ for any $a,b,c \geq 0$ and using Burkholder-Davis-Gundy inequality, the $p^{th}$ moment of \eqref{JNA-17} becomes
\begin{align}\label{JNA-41}
\E|u_t^{(\gamma)}(x)-&u_t^{(\beta)}(x)|^p \leq 3^p|(\mathcal{G}_B^{(\gamma)} u)_t(x)-(\mathcal{G}_B^{(\beta)}u)_t(x)|^p \nonumber\\
&+(3\lambda)^p\left[\int_0^t\int_B|G_B^{(\gamma)}(t-s,\,x,\,y)-G_B^{(\beta)}(t-s,\,x,\,y)|^2[\E|\sigma(u_s^{(\gamma)}(y))|^p]^{2/p}\,\d y\,\d s\right]^{p/2}\\
&+(3\lambda)^p\left[\int_B \int_0^tG_B^{(\beta)}(t-s,\,x,\,y)^2[\E[\sigma(u_s^{(\gamma)}(y))-\sigma(u_s^{(\beta)}(y))]^p]^{2/p}\d s\,\d y \right]^{p/2}.\nonumber
\end{align}
For a fixed $\theta>0$ the inequality in \eqref{JNA-41} becomes
\begin{align*}
e^{-\theta t} \E|u_t^{(\gamma)}(x)-&u_t^{(\beta)}(x)|^p \lesssim e^{-\theta t} |(\mathcal{G}_B^{(\gamma)} u)_t(x)-(\mathcal{G}_B^{(\beta)}u)_t(x)|^p \nonumber\\
&+e^{-\theta t}\left[\int_0^t\int_B|G_B^{(\gamma)}(t-s,\,x,\,y)-G_B^{(\beta)}(t-s,\,x,\,y)|^2[\E|\sigma(u_s^{(\gamma)}(y))|^p]^{2/p}\,\d y\,\d s\right]^{p/2}\\
&+e^{-\theta t}\left[\int_B \int_0^tG_B^{(\beta)}(t-s,\,x,\,y)^2[\E[\sigma(u_s^{(\gamma)}(y))-\sigma(u_s^{(\beta)}(y))]^p]^{2/p}\d s\,\d y \right]^{p/2}.\\
&:=J_1+J_2+J_3.
\end{align*}
Let us first work on the third term.
Using Lemma \ref{Lemma-2} and since $\sigma$ is globally Lipschitz, we have
\begin{align} \label{JNA-20}
J_3 &:= e^{-\theta t}\left[\int_B \int_0^tG_B^{(\beta)}(t-s,\,x,\,y)^2[\E[\sigma(u_s^{(\gamma)}(y))-\sigma(u_s^{(\beta)}(y))]^p]^{2/p}\d s\,\d y \right]^{p/2} \nonumber\\
&\leq  L_{\sigma}^p \left [\int_B\int_0^tG_B^{(\beta)}(t-s,\,\text{x},\,\text{y})^2 e^{\frac{-2\theta t}{p}}[\E|u_s^{(\gamma)}(\text{y})-u_s^{(\beta)}(\text{y})|^p]^{2/p}\d s\,\d y \right]^{p/2} \nonumber\\
&\leq  L_{\sigma}^p b(\theta) \left [\int_B\int_0^tG_B^{(\beta)}(t-s,\,\text{x},\,\text{y})^2 e^{\frac{-2\theta (t-s)}{p}} \d s\,\d y \right]^{p/2}\nonumber\\
&\leq  L_{\sigma}^p b_t(\theta) \left [\int_0^\infty\int_{\R^n}G^{(\beta)}(s,\,\text{x},\,\text{y})^2 e^{\frac{-2\theta s}{p}} \d y\,\d s \right]^{p/2}\nonumber\\
 &\lesssim  L_{\sigma}^p b_t(\theta) \left [\int_0^\infty s^{-\frac{\beta d}{\alpha}} e^{\frac{-2\theta s}{p}} \,\d s \right]^{p/2}\nonumber\\
 &\lesssim  L_{\sigma}^p b_t(\theta) \left [\int_0^\infty s^{-\frac{\beta d}{\alpha}} e^{-\theta s} \,\d s \right]^{p/2}\nonumber\\
&\lesssim  L_{\sigma}^p b_t(\theta) \left[\Gamma\left (1-\frac{\beta d}{\alpha}\right)\theta ^{\frac{\beta d}{\alpha}-1}\right]^{p/2},
\end{align}
where $b_t(\theta)=\sup_{0<s<t}\sup_{\text{z}\in B}e^{-\theta s}\E|u_s^{(\gamma)}(\text{z})-u_s^{(\beta)}(\text{z})|^p$.
Now we fix $\theta > 0$ sufficiently large so that
\begin{align} \label{JNA-21}
J_3 \lesssim \frac{1}{2} b_t(\theta).
\end{align}
Next we work on $J_2$.  Since $(a + b)^2 \leq 4(a^2+b^2)$ for $a,b\geq0$, using Proposition \ref{JNA-42}, Lemma \ref{nane-tuan}(b) and the assumption on $\sigma$, we obtain
\begin{align} \label{JNA-21E}
J_2&:=e^{-\theta t}\left[\int_0^t\int_B|G_B^{(\gamma)}(t-s,\,x,\,y)-G_B^{(\beta)}(t-s,\,x,\,y)|^2[\E|\sigma(u_s^{(\beta)}(y))|^p]^{2/p}\,\d y\,\d s\right]^{p/2} \nonumber\\
&\lesssim \sup_{t>0,\,x\in B}e^{-\theta t}\E|u_s^{(\beta)}(x)|^p\left[\int_0^\infty\int_B e^{-2\theta s/p}|G_B^{(\gamma)}(s,\,x,\,y)-G_B^{(\beta)}(s,\,x,\,y)|^2\,\d s\,\d y\right]^{p/2} \nonumber\\
&\lesssim \left[\int_0^\infty\int_B e^{-2\theta s/p}|G_B^{(\gamma)}(s,\,x,\,y)-G_B^{(\beta)}(s,\,x,\,y)|^2\,\d s\,\d y\right]^{p/2} \nonumber \\
&\lesssim \left[\int_0^\infty\int_B e^{-2\theta s/p}|G_B^{(\gamma)}(s,\,x,\,y)|^2\d s\,\d y+\int_0^\infty\int_B e^{-2\theta s/p} |G_B^{(\beta)}(s,\,x,\,y)|^2\,\d s\,\d y\right]^{p/2} \nonumber \\
 &\lesssim \left[ p^{1-2\lambda^{'} \beta_0^{'}}\left [ \frac{\mathscr C_{\mu_0^{'}}^2 \theta^{-2\mu_0^{'}}}{1-2\lambda^{'} -2\mu_0}  +  \frac{\mathscr C_{\mu_1^{'}}^2 \theta^{-2\mu_1^{'}}}{2\lambda^{'} \beta_0^{'} +2\mu_1^{'}-1} \right] + p^{1-2\lambda \beta_0}\left [\frac{\mathscr C_{\mu_0}^2 \theta^{-2\mu_0}}{1-2\lambda - 2\mu_0} +  \frac{\mathscr C_{\mu_1}^2 \theta^{-2\mu_1}}{2\lambda \beta_0 +2\mu_1-1} \right ]  \right ]^{p/2},
\end{align}
where $\beta\in (\beta_0, \beta_1)$ and $\gamma\in (\beta_0^{'}, \beta_1^{'})$ for some $\beta_0, \beta_1, \beta_0^{'}, \beta_1^{'} \in (0,1)$,  $\frac{d}{\alpha}<\lambda<\frac{1}{2\beta}$, $\frac{d}{\alpha}<\lambda ^{'}<\frac{1}{2\gamma}$, \\ $\mu_0 < \min\Big(  \frac{1}{2}-\lambda,~~\frac{1-2\beta_1 \lambda}{2}\Big)$, $\mu_0^{'} < \min\Big( \frac{1}{2} -\lambda^{'},~~\frac{1-2\gamma_1 \lambda^{'}}{2}\Big)$, $\mu_1> \frac{1- 2 \beta_0 \lambda}{2}$,    and $\mu_1^{'}> \frac{1- 2 \beta_0^{'} \lambda^{'}}{2}$, where the bound on right hand side of \eqref{JNA-21E} doesn't depend on $\beta$ and $\gamma$.\\
We combine the above estimates and obtain the following:
\begin{align*}
\sup_{x\in B}e^{-\theta s} \E|u_s^{(\gamma)}(x)-u_s^{(\beta)}(x)|^p&\lesssim  e^{-\theta t} |(\mathcal{G}_B^{(\gamma)} u)_t(x)-(\mathcal{G}_B^{(\beta)}u)_t(x)|^p + \frac{1}{2} b_t(\theta)\\
&+ \left[\int_0^\infty\int_B e^{-2\theta s/p}|G_B^{(\gamma)}(s,\,x,\,y)-G_B^{(\beta)}(s,\,x,\,y)|^2\,\d s\,\d y\right]^{p/2}.
\end{align*}
Since Lemma \ref{JNA-27}, and Lemma \ref{nane-tuan}(b) allows us to use the dominated convergence theorem,  taking $\gamma\rightarrow \beta$ the result of the theorem follows.
\end{proof}

\section{Proof of Theorem \ref{Thrm-5}, Theorem \ref{Thrm-6}, and Theorem \ref{Thrm-7}} \label{proof3}

\begin{proof}[\bf Proof of Theorem \ref{Thrm-5}]

First, we establish the second statement of the theorem. From the mild formulation given in \eqref{JNA-9} and using Stochastic Fubini theorem, we set
\begin{align} \label{JNA-c10}
\langle u_t, \varphi_{1}\rangle
&=E_\beta(-\mu_{1} t^\beta)\langle u_0, \varphi_{1}\rangle\, +\lambda \int_B   \int_0^t E_\beta(-\mu_{1} (t-s)^\beta)\varphi_{1}(\text{y})\sigma(v_s(\text{y}))F(\d s\,,\d \text{y}).
\end{align}
Taking the second moment of \eqref{JNA-c10}, we get
\begin{align} \label{JNA-c11}
\E\langle u_t, \varphi_{1}\rangle^2&=E_\beta(-\lambda_1 t^\beta)^2 \langle u_0, \varphi_{1}\rangle^2\nonumber\\ &~~~~~~+\lambda^2\int_{B \times B}\int_0^t E_\beta(-\lambda_1(t-s)^\beta)^2\varphi_{1}(\text{y})\varphi_{1}(\text{z})\E|\sigma(u_s(\text{y}))\sigma(u_s(\text{z}))|\,f(y,\,z)\,\d s\d \text{y}\d \text{z}.
\end{align}
Now using \eqref{JNA-c11} and the assumption on $\sigma$ gives
\begin{align}\label{JNA-c12}
\int_0^\infty e^{-\theta t}\E\langle u_t, \varphi_{1}\rangle^2\,\d t
&= \Lambda (\theta) \,\langle u_0, \varphi_{1}\rangle^2 +\lambda^2 \Lambda (\theta) \,\int_0^\infty e^{-\theta t}\int_{B \times B} \varphi_{1}(\text{y})\varphi_{1}(\text{z}) \times \nonumber \\ &~~~~~~~~~~~~~~~~~~~~~~~~~~~~~~~~~~\E|\sigma(u_t(\text{y}))\sigma(u_t(\text{z}))| f(y,\,z) \d \text{y} \,\text{z}\,\d t. \nonumber \\
&\geq \Lambda (\theta) \,\langle u_0, \varphi_{1}\rangle^2 +\lambda^2l_\sigma^2 K_f \Lambda (\theta) \,\int_0^\infty e^{-\theta t}\int_{B \times B} \varphi_{1}(\text{y})\varphi_{1}(\text{z})\times \nonumber \\ &~~~~~~~~~~~~~~~~~~~~~~~~~~~~~~~~~~\E|u_t(\text{y})u_t(\text{z})| \d \text{y} \, \d \text{z}\,\d t \nonumber \\
&=  \Lambda(\theta) \langle u_0, \varphi_{1}\rangle^2+\lambda^2l_\sigma^2 K_f \Lambda(\theta)\int_0^\infty e^{-\theta t}\E\langle u_t, \varphi_{1}\rangle^2\,\d t.
\end{align}
If $\beta\in(0,\,\frac{1}{2}]$, since $\Lambda(\theta)$ tends to infinity as $\theta$ goes to zero, we can choose $\theta$ small enough and from \eqref{JNA-c12} we obtain
\begin{align}\label{JNA-c12b}
\int_0^\infty e^{-\theta t}\E\langle u_t, \varphi_{1}\rangle^2\,\d t &\gtrsim  \Lambda(\theta) \langle u_0, \varphi_{1}\rangle^2+\frac{1}{2}\int_0^\infty e^{-\theta t}\E\langle u_t, \varphi_{1}\rangle^2\,\d t.
\end{align}
This implies that for sufficiently small $\theta$, we have
\begin{equation*}
\int_0^\infty e^{-\theta t}\E\langle u_t, \varphi_{1}\rangle^2\,\d t=\infty,
\end{equation*}
which also implies that for sufficiently large $t$, $\E\langle u_t, \varphi_{1}\rangle^2$ grows exponentially. Now using Cauchy-Schwarz inequality we get the following,
\begin{align*}
\E\langle u_t, \varphi_{1}\rangle^2&=\E \left |\int_B u_t(x) \,\varphi_1(x)\,\d x \right|^2\\
&\leq \E \left ( \left(\int_B |u_t(x)|^2 \,\d x \right) \left (\int_B \varphi_1^2(x)\,\d x \right)\right)\\
&\lesssim  \sup_{\text{x}\in B}\E|u_t(\text{x})|^2.
\end{align*}
We can thus conclude that $\sup_{\text{x}\in B}\E|u_t(\text{x})|^2$ too grows exponentially fast for all values of $\lambda$. \\
Note that the first part of the conclusion of theorem merely follows from the fact that the second term of \eqref{JNA-c12} is positive and that the first term cannot have exponential decay.
\end{proof}

\begin{proof}[\bf Proof of  Theorem \ref{Thrm-6}]
Using the same notations as in the proof of Theorem \ref{Thrm-5}, we observe that the inequality \eqref{JNA-c12} holds;
\begin{align}\label{JNA-c13}
\int_0^\infty e^{-\theta t}\E\langle u_t, \varphi_{1}\rangle^2\,\d t&\gtrsim  \Lambda(\theta) \langle u_0, \varphi_{1}\rangle^2+\lambda^2l_\sigma^2K_f\Lambda(\theta)\int_0^\infty e^{-\theta t}\E\langle u_t,\,\varphi_{1}\rangle^2\,\d t.
\end{align}
Since $2\beta>1$, the function $\Lambda(\theta)$ is bounded. Thus, for any fixed $\theta>0$ there exists sufficiently large $\lambda_u$ so that for all $\lambda\geq \lambda_u$, from the inequality \eqref{JNA-c13} we obtain
\begin{align} \label{JNA-c14}
\int_0^\infty e^{-\theta t}\E\langle u_t, \varphi_{1}\rangle^2\,\d t&\gtrsim  \Lambda(\theta) \langle u_0, \varphi_{1}\rangle^2+2\int_0^\infty e^{-\theta t}\E\langle u_t, \varphi_{1}\rangle^2\,\d t.
\end{align}
Again, following similar lines of arguments as in the proof of Theorem \ref{Thrm-5}, we get result of the current theorem.
\end{proof}

\begin{proof}[\bf Proof of  Theorem \ref{Thrm-7}]
Using $(a+b)^2 \leq 4(a^2+b^2)$ for $a,b \geq 0$, and applying Burkholder's inequality, the second moment of the mild formulation given by \eqref{JNA-9-b} becomes:
\begin{align} \label{JNA-c15}
\E|u_t(x)|^2&\leq
4|(\mathcal{G}_B^{(\beta)} u)_t(\text{x})|^2+ 4\lambda^2 \int_{B\times B} \int_0^t G_B^{(\beta)}(t-s,\,\text{x},\,\text{y})G_B^{(\beta)}(t-s,\,\text{x},\,\text{z}) \times \nonumber ~
\\ & ~~~~~~~~~~~~~~~~~~~~~~~~~~~~~~\E|\sigma(u_s(\text{y}))\sigma(u_s(\text{z}))|\,f(y,\, z)\,\d s\,\d \text{y} \,\d \text{z}\nonumber\\
&\leq 4|(\mathcal{G}_B^{(\beta)} u)_t(\text{x})|^2+ 4\lambda^2 L_\sigma^2\int_{B\times B}\int_0^t G_B^{(\beta)}(t-s,\,\text{x},\,\text{y})G_B^{(\beta)}(t-s,\,\text{x},\,\text{z})\times \nonumber ~
\\ & ~~~~~~~~~~~~~~~~~~~~~~~~~~~~~~\E|u_s(\text{y})u_s(\text{z})| \, f(y,\,z)\,\d s\,\d \text{y} \, \d \text{z}\nonumber\\
&:=J_1+J_2.
\end{align}
We observe $J_1$ is bounded by the square of the same constant as initial condition, since $\int_B G_B^{(\beta)}(t,\,x)\,\d x \leq \int_{\R^d} G^{(\beta)}(t,\,x)\,\d x =1$ and the initial datum is assumed to be bounded above by a constant. It remains to bound $J_2$. Now if $\{\varphi_n\}_{n\geq1}$ are uniformly bounded by a constant $C(B)$,for each fixed $t > 0$, by Cauchy-Schwarz inequality, Lemma \ref{Lema-1} and Lemma \ref{JNA-bd1} we obtain
\begin{align} \label{JNA-c16}
J_2&:=4\lambda^2 L_\sigma^2\int_{B\times B}\int_0^t G_B^{(\beta)}(t-s,\,\text{x},\,\text{y})G_B^{(\beta)}(t-s,\,\text{x},\,\text{z})\E|u_s(\text{y})u_s(\text{z})| \, f(y,\,z)\,\d s\,\d \text{z} \, \d \text{y}\nonumber\\
&\leq 4\lambda^2 L_\sigma^2\int_{B\times B}\int_0^t G_B^{(\beta)}(t-s,\,\text{x},\,\text{y})G_B^{(\beta)}(t-s,\,\text{x},\,\text{z})(\E|u_s(\text{y})|^2)^{\frac{1}{2}}(\E|u_s(\text{z})|^2)^{\frac{1}{2}} \, f(y,\,z)\,\d s\,\d \text{z} \, \d \text{y}\nonumber\\
&\leq 4\lambda^2 L_\sigma^2\int_0^t \sup_{x \in B} \E[|u_s(\text{x})|^2] \int_{B\times B} G_B^{(\beta)}(t-s,\,\text{x},\,\text{y})G_B^{(\beta)}(t-s,\,\text{x},\,\text{z}) \, f(y,\,z)\,\d s\,\d \text{z} \, \d \text{y}\nonumber\\
&\leq 4\lambda^2 L_\sigma^2 a_t\int_0^t \sum_{n=1}^\infty \sum_{k=1}^\infty E_\beta(-\mu_{n}(t-s)^\beta) E_\beta(-\mu_{k}(t-s)^\beta) \varphi_{n}(\text{x})\varphi_{k}(\text{x}) \times \nonumber\\
&~~~~~~~~~~~~~~~~~~~~\int_{B\times B}\varphi_{n}(\text{y})\varphi_{k}(\text{z}) \, f(y,\,z)\,\d s\,\d \text{z} \, \d \text{y}\nonumber\\
&\leq 4\lambda^2 L_\sigma^2 a_t \sum_{n=1}^\infty \sum_{k=1}^\infty \left (\int_0^t E_\beta(-\mu_{n}(t-s)^\beta)^2 \, \d s \right )^{\frac{1}{2}} \left (\int_0^t E_\beta(-\mu_{k}(t-s)^\beta)^2 \, \d s \right )^{\frac{1}{2}} \varphi_{n}(\text{x})\varphi_{k}(\text{x}) \,\times \nonumber\\
&~~~~~~~~~~~~~~~~~~~~\int_{B}\varphi_{n}(\text{y})\left [\int_B \varphi_{k}(\text{z}) \, f(y,\,z) \,\d \text{z} \right ] \, \d \text{y} \nonumber\\
    &\lesssim C_3 C(B)^4 \lambda^2 L_\sigma^2 a_t \sum_{n=1}^\infty \sum_{k=1}^\infty \left (\mu_n^{-\frac{1}{2\beta}} \right ) \left (\mu_k^{-\frac{1}{2\beta}} \right ) \,\int_{B}\left [\int_B f(y,\,z) \,\d \text{z} \right ] \, \d \text{y} \nonumber\\
    &\lesssim C_3\lambda^2 L_\sigma^2 a_t \left (\sum_{n=1}^\infty n^{-\frac{\alpha}{2d\beta}} \right )  \left (\sum_{k=1}^\infty  k^{-\frac{\alpha}{2d\beta}} \right ) \,\int_{B \times B} f(y,\,z) \,\d \text{z} \,\d \text{y} \nonumber \\
    &\lesssim C \lambda^2 L_\sigma^2 a_t,
\end{align}
where $C=C_3\left (\sum_{n=1}^\infty n^{-\frac{\alpha}{2d\beta}} \right )  \left (\sum_{k=1}^\infty  k^{-\frac{\alpha}{2d\beta}} \right ) \,\int_{B \times B} f(y,\,z) \,\d \text{z} \,\d \text{y} <\infty$,  and~\\ $a_t = \sup_{0<s<t} \sup_{x \in B} \E[|u_s(\text{x})|^2$.
Since $\beta\in (\frac{1}{2}, 1)$,  we can choose $\lambda_l$ sufficiently small that for all $\lambda\leq \lambda_l$, the above estimates shows
\begin{align*}
\sup_{0<s<t}\sup_{\text{x}\in B} \E|u_t(\text{x})|^2\lesssim 1+\frac{1}{2}\sup_{0<s<t}\sup_{\text{x}\in B} \E|u_t(\text{x})|^2.
\end{align*}
This shows that $\sup_{0<t<\infty}\sup_{\text{x}\in B} \E|u_t(\text{x})|^2$ is finite and hence the conclusion of the theorem follows.
\end{proof}

\begin{proof} [\bf Proof of  Corollary \ref{Thrm-8}]
Since $\frac{1}{\beta}>1$, if $\alpha = 2$ and $d=1$, (for the case of the eigenvalues $\mu_n=(\frac{n\pi}{L})^2$ and eigenvectors $\varphi_n(x)=(\frac{2}{L})^{\frac{1}{2}}\sin{(\frac{n\pi x}{L})}$ of Dirichlet Laplacian corresponding to a Brownian Motion killed upon
exiting the domain $[0, \,L]$ so that $|\varphi_n(x)| \leq (\frac{2}{L})^{\frac{1}{2}}$ for each $n$) from \eqref{JNA-16} using Lemma \ref{Lema-1} we obtain the result of the corollary.
\end{proof}
\begin{proof} [\bf Proof of  Corollary \ref{Thrm-9}]
In this case, we note that $ C_4\,\int_{B}\left [\int_B f(y,\,z) \,\d \text{z} \right ] \, \d \text{y} < \infty$, since
\begin{align}
    \int_{B\times B} f(y,\,z) \,\d \text{z} \, \d \text{y}
     &=2 \frac{L^{2-\gamma}}{(1-\gamma)\,(2-\gamma)} < \infty.
\end{align}
\end{proof}
\section{Proof of Theorem \ref{Thrm-10}}\label{proof5}
In this section we give the proof of Theorem \ref{Thrm-10} and the following result is useful to guarantee  the use of dominated convergence theorem in the proof of Theorem \ref{Thrm-10}.
\begin{proposition} \label{nane-tuan-2022-b}
Assume that $\{\varphi_n\}_{n\geq1}$ are uniformly bounded by a constant $C(B)$, depending on the geometrical characteristics of the domain $B$, i.e.,
\begin{equation*}
    C(B) = \sup_{n\geq 1,\text{x} \in B} \varphi_n(\text{x}).
\end{equation*}
Suppose that $\int_{B \times B} f(x,\,y)\, \d x \,\d y <\infty$. For $\beta\in (\frac{1}{2},1)$ and   $\frac{d}{\alpha}<\lambda<\frac{1}{2\beta}$,
let $0< \mu_0 < \min\Big(\frac{1}{2} - \lambda,~~\frac{1-2\beta_1 \lambda}{2} \Big)<\min\Big(\frac{1}{2} - \lambda,~~\frac{1-2\beta \lambda}{2} \Big)$ and $\mu_1> \frac{1- 2 \beta_0 \lambda}{2}$ where $\beta\in (\beta_0,\, \beta_1)$ for some $\beta_0, \beta_1 \in (0,\,1)$. Also for $\gamma\in (\frac{1}{2},1)$ and   $\frac{d}{\alpha}<\lambda^{'}<\frac{1}{2\gamma}$, let $0< \mu_0 ^{'}< \min\Big(\frac{1}{2} - \lambda^{'},~~\frac{1-2\beta_1^{'} \lambda^{'}}{2} \Big)<\min\Big(\frac{1}{2} - \lambda^{'},~~\frac{1-2\gamma \lambda^{'}}{2} \Big)$,  and $\mu_1^{'}> \frac{1- 2 \beta_0^{'} \lambda^{'}}{2}$  where $\gamma\in (\beta_0^{'}, \,\beta_1^{'})$ for some $\beta_0^{'}, \beta_1^{'} \in (0,\,1)$. Then for $d<\frac{1}{2} \min\{\frac{1}{\beta},\,\frac{1}{\gamma}\}\alpha$, there exist positive a constant $C$ that depends on $\mu_0, \,\mu_0^{'}, \,\mu_1,\, \mu_1^{'},$ and $p$, and  independent of $\beta$ and $\gamma$ such that
\begin{align} \label{a13}
\bigg[ \int_0^t  &\int_{B\times B} e^{-2\theta (t-s)/p}|G_B ^{(\gamma)}(t-s,\,x,\,y)-G_B^{(\beta)}(t-s,\,x,\,y)||G_B^{(\gamma)}(t-s,\,x,\,z)-G_B^{(\beta)}(t-s,\,x,\,z)|\times \nonumber \\ & f(y,\,z)\,\d y \, \d z \,\d s \,\bigg]^\frac{p}{2}
\lesssim C \left [p^{1-2\lambda^{'} \beta_0^{'}}\left (\theta^{-2\mu_0^{'}} +  \theta^{-2\mu_1^{'}} \right ) +  p^{1-2\lambda \beta_0}\left ( \theta^{-2\mu_0}  +  \theta^{-2\mu_1} \right) \right]^{p/2}.
\end{align}
\end{proposition}
\begin{proof} Since $\{\varphi_n\}_{n\geq1}$ are uniformly bounded by a constant $C(B)$ and $\int_{B \times B} f(x,\,y)\, \d x \,\d y <\infty$, for each $t>0$, using Lemma \ref{Integral-bound}, Lemma \ref{nane-tuan}(a)  and the fact that $(a + b)^2 \leq 2(a^2+b^2)$ for $a,b\geq0$, we bound the integral term as follows:
\begin{align} \label{G5}
   \int_0^t  &\int_{B\times B} e^{-2\theta (t- s)/p}|G_B ^{(\gamma)}(t-s,\,x,\,y)-G_B^{(\beta)}(t-s,\,x,\,y)||G_B^{(\gamma)}(t-s,\,x,\,z)-G_B^{(\beta)}(t-s,\,x,\,z)| \times \nonumber \\&~~~~~~~~~~~~~~~~ f(y,\,z)\,\d y \, \d z \,\d s \nonumber\\
&\lesssim \int_0^t  \int_{B\times B} e^{-2\theta (t-s)/p} [(t-s)^{-2\gamma \lambda^{'}} +2(t-s)^{-(\beta \lambda +\gamma \lambda^{'})}+(t-s)^{-2\beta \lambda}]f(y,\,z)\,\d y \, \d z \,\d s \nonumber\\
&\lesssim \int_0^\infty  \int_{B\times B} e^{-2\theta s/p} [s^{-\gamma \lambda^{'}}+s^{-\beta \lambda}]^2f(y,\,z)\,\d y \, \d z \,\d s \nonumber\\
&\lesssim \int_0^\infty   e^{-2\theta s/p} [s^{-2\gamma \lambda^{'}}+s^{-2\beta \lambda}] \,\d s \nonumber\\
&\lesssim  p^{1-2\lambda^{'} \beta_0^{'}}\left [\frac{\mathscr C_{\mu_0^{'}}^2 \theta^{-2\mu_0^{'}}}{1-2\lambda^{'} - 2\mu_0^{'}} +  \frac{\mathscr C_{\mu_1^{'}}^2 \theta^{-2\mu_1^{'}}}{2\lambda^{'} \beta_0^{'} +2\mu_1^{'}-1} \right ] +  p^{1-2\lambda \beta_0}\left [ \frac{\mathscr C_{\mu_0}^2 \theta^{-2\mu_0}}{1-2\lambda -2\mu_0}  +  \frac{\mathscr C_{\mu_1}^2 \theta^{-2\mu_1}}{2\lambda \beta_0 +2\mu_1-1} \right],
\end{align}
where $\mu_0$ and $\mu_0^{'}$ are as given in the hypothesis of the proposition. It  follows from \eqref{G5} that for any $p \ge 0$ we obtain
\begin{align*}
\bigg[ \int_0^t  &\int_{B\times B} e^{-2\theta (t- s)/p}|G_B ^{(\gamma)}(t-s,\,x,\,y)-G_B^{(\beta)}(t-s,\,x,\,y)||G_B^{(\gamma)}(t-s,\,x,\,z)-G_B^{(\beta)}(t-s,\,x,\,z)|\times \nonumber \\& f(y,\,z)\,\d y \, \d z \,\d s \bigg]^{p/2} \\ &\lesssim \left [p^{1-2\lambda^{'} \beta_0^{'}}\left [\frac{\mathscr C_{\mu_0^{'}}^2 \theta^{-2\mu_0^{'}}}{1-2\lambda^{'}  - 2\mu_0^{'}} +  \frac{\mathscr C_{\mu_1^{'}}^2 \theta^{-2\mu_1^{'}}}{2\lambda^{'} \beta_0^{'} +2\mu_1^{'}-1} \right ] +  p^{1-2\lambda \beta_0}\left [ \frac{\mathscr C_{\mu_0}^2 \theta^{-2\mu_0}}{1-2\lambda -2\mu_0}  +  \frac{\mathscr C_{\mu_1}^2 \theta^{-2\mu_1}}{2\lambda \beta_0 +2\mu_1-1} \right] \right]^{p/2},
\end{align*}
which allows us to deduce \eqref{a13}.
That is, each terms in the summation can be bounded by a quantity independent of $\gamma$ and $\beta$.\\
\end{proof}

\begin{proof}[\bf Proof of Theorem \ref{Thrm-10}]
From the mild formulation of the solutions, we have
\begin{align}\label{JNA-c17}
u_t^{(\gamma)}(x)-u_t^{(\beta)}(x)&=(\mathcal{G}_B^{(\gamma)} u)_t(x)-(\mathcal{G}_B^{(\beta)}u)_t(x)+\lambda \int_B\int_0^tG_B^{(\gamma)}(t-s,\,x,\,y)\sigma(u_s^{(\gamma)}(y))F(\d s\,\d y) \nonumber\\
&-\lambda \int_B\int_0^tG_B^{(\beta)}(t-s,\,x,\,y)\sigma(u_s^{(\beta)}(y))F(\d s\,\d y)\nonumber\\
&=(\mathcal{G}_B^{(\gamma)} u)_t(x)-(\mathcal{G}_B^{(\beta)}u)_t(x) \nonumber\\
&+\lambda\int_B\int_0^t[G_B^{(\gamma)}(t-s,\,x,\,y)-G_B^{(\beta)}(t-s,\,x,\,y)]\sigma(u_s^{(\gamma)}(y))F(\d s\,\d y)\\
&+\lambda\int_B\int_0^tG_B^{(\beta)}(t-s,\,x,\,y)[\sigma(u_s^{(\gamma)}(y))-\sigma(u_s^{(\beta)}(y))]F(\d s\,\d y).\nonumber
\end{align}
Now applying the inequality $(a+b+c)^p \leq 3^p(a^p\,+b^p\,+c^p)$ for any $a,b,c \geq 0$ and using Burkholder-Davis-Gundy inequality, the $p^{th}$ moment of \eqref{JNA-c17} becomes
\begin{align}\label{JNA-c41}
\E|u_t^{(\gamma)}(x)-&u_t^{(\beta)}(x)|^p \leq 3^p|(\mathcal{G}_B^{(\gamma)} u)_t(x)-(\mathcal{G}_B^{(\beta)}u)_t(x)|^p \nonumber\\
&+(3\lambda)^p \E\left|\int_B\int_0^t[G_B^{(\gamma)}(t-s,\,x,\,y)-G_B^{(\beta)}(t-s,\,x,\,y)]\sigma(u_s^{(\gamma)}(y))F(\d s\,\d y) \right |^p\\
&+(3\lambda)^p \E\left |\int_B\int_0^tG_B^{(\beta)}(t-s,\,x,\,y)[\sigma(u_s^{(\gamma)}(y))-\sigma(u_s^{(\beta)}(y))]F(\d s\,\d y) \right|^p.\nonumber
\end{align}
\begin{align}\label{JNA-c41-b}
\E|u_t^{(\gamma)}(x)-&u_t^{(\beta)}(x)|^p \leq 3^p|(\mathcal{G}_B^{(\gamma)} u)_t(x)-(\mathcal{G}_B^{(\beta)}u)_t(x)|^p \nonumber\\
&+(3\lambda)^p\bigg[\int_0^t\int_{B\times B}|G_B^{(\gamma)}(t-s,\,x,\,y)-G_B^{(\beta)}(t-s,\,x,\,y)||G_B^{(\gamma)}(t-s,\,x,\,z)-G_B^{(\beta)}(t-s,\,x,\,z)|\times \nonumber\\
&~~~~~~~~~~[\E|\sigma(u_s^{(\gamma)}(y))\sigma(u_s^{(\gamma)}(z))|^{p/2}]^{2/p}  \,f(y,\,z)\,\d s\, \d y\,\d z \bigg ]^{p/2}\\
&+(3\lambda)^p\bigg[\int_{B \times B} \int_0^tG_B^{(\beta)}(t-s,\,x,\,y)G_B^{(\beta)}(t-s,\,x,\,z)  \times \nonumber\\
& ~~~~~~~~~~[\E[(\sigma(u_s^{(\gamma)}(y))-\sigma(u_s^{(\beta)}(y)))(\sigma(u_s^{(\gamma)}(z))-\sigma(u_s^{(\beta)}(z)))]^{p/2}]^{2/p} \,f(y,\,z)\,\d s\,\d y \,\d z \bigg]^{p/2}.\nonumber\\
\end{align}
For a fixed $\theta>0$ the inequality in \eqref{JNA-c41} becomes
\begin{align}\label{JNA-c41a}
 e^{-\theta t} &\E|u_t^{(\gamma)}(x)-u_t^{(\beta)}(x)|^p \lesssim  e^{-\theta t} |(\mathcal{G}_B^{(\gamma)} u)_t(x)-(\mathcal{G}_B^{(\beta)}u)_t(x)|^p \nonumber\\
&+ e^{-\theta t} \bigg[\int_0^t\int_{B \times B}|G_B^{(\gamma)}(t-s,\,x,\,y)-G_B^{(\beta)}(t-s,\,x,\,y)||G_B^{(\gamma)}(t-s,\,x,\,z)-G_B^{(\beta)}(t-s,\,x,\,z)|\times \nonumber\\
&~~~~~~~~~~[\E|\sigma(u_s^{(\gamma)}(y))\sigma(u_s^{(\gamma)}(z))|^{p/2}]^{2/p} \,f(y,\,z)\,\d s\, \d y\,\d z \bigg ]^{p/2} \nonumber\\
&+ e^{-\theta t} \bigg[\int_{B \times B} \int_0^tG_B^{(\beta)}(t-s,\,x,\,y)G_B^{(\beta)}(t-s,\,x,\,z)  \times \nonumber\\
& ~~~~~~~~~~[\E[(\sigma(u_s^{(\gamma)}(y))-\sigma(u_s^{(\beta)}(y)))(\sigma(u_s^{(\gamma)}(z))-\sigma(u_s^{(\gamma)}(z)))]^{p/2}]^{2/p} \,f(y,\,z)\,\d s\,\d y \,\d z \bigg]^{p/2}.\nonumber\\
&:=J_1+J_2+J_3.
\end{align}
Let us first work on the third term.
Since $\sigma$ is globally Lipschitz, using Lemma \ref{JNA-45}, and Lemma  \ref{nane-tuan} (a), we have

\begin{align} \label{JNA-c20-a}
J_3 &:= e^{-\theta t} \bigg[\int_{B \times B} \int_0^t G_B^{(\beta)}(t-s,\,x,\,y)G_B^{(\beta)}(t-s,\,x,\,z)  \times  \nonumber\\
& ~~~~~~~~~~\bigg[\E[(\sigma(u_s^{(\gamma)}(y))-\sigma(u_s^{(\beta)}(y)))(\sigma(u_s^{(\gamma)}(z))-\sigma(u_s^{(\gamma)}(z)))]^{p/2}]^{2/p} \,f(y,\,z)\,\d s\,\d y \,\d z \bigg]^{p/2}.\nonumber\\
&\lesssim L_{\sigma}^p  \bigg[\int_{B \times B} \int_0^t G_B^{(\beta)}(t-s,\,x,\,y)G_B^{(\beta)}(t-s,\,x,\,z)  \times \nonumber\\
& ~~~~~~~~~~e^{\frac{-2\theta t}{p}}[\E[(u_s^{(\gamma)}(y)-u_s^{(\beta)}(y))(u_s^{(\gamma)}(z)-u_s^{(\beta)}(z))]^{p/2}]^{2/p}  \,f(y,\,z)\,\d s\,\d y \,\d z \bigg]^{p/2}.\nonumber\\
&\lesssim L_{\sigma}^p \bigg[\int_{B \times B} \int_0^t G_B^{(\beta)}(t-s,\,x,\,y)G_B^{(\beta)}(t-s,\,x,\,z)  \times \nonumber\\
& ~~~~~~~~~~e^{\frac{-2\theta t}{p}}(\E(u_s^{(\gamma)}(y)-u_s^{(\beta)}(y))^{p})^{\frac{1}{p}} (\E(u_s^{(\gamma)}(z)-u_s^{(\beta)}(z))^{p})^{\frac{1}{p}} \,f(y,\,z)\,\d s\,\d y \,\d z \bigg]^{p/2}.\nonumber\\
&\lesssim  L_{\sigma}^p b_t(\theta) \left[\int_{B \times B} \int_0^t G_B^{(\beta)}(t-s,\,x,\,y)G_B^{(\beta)}(t-s,\,x,\,z)  e^{\frac{-2\theta (t-s)}{p}}\,f(y,\,z)\,\d s\,\d y \,\d z\right]^{p/2} \nonumber\\
&\lesssim  L_{\sigma}^p b_t(\theta) \left[\int_0^t (t-s)^{-2\beta \lambda}  e^{\frac{-2\theta (t-s)}{p}}\,\d s\,\int_{B \times B} f(y,\,z)\d y \,\d z\right]^{p/2} \nonumber\\
&\lesssim  L_{\sigma}^p b_t(\theta) \left[\int_0^\infty s^{-2\beta \lambda}  e^{\frac{-2\theta s}{p}}\,\d s\,\int_{B \times B} f(y,\,z)\d y \,\d z\right]^{p/2} \nonumber\\
&\lesssim  L_{\sigma}^p p^{1-2 \beta \lambda} b_t(\theta) \left[\int_0^\infty \frac{e^{-2\theta r}}{r^{2\beta \lambda} } \,\d r\,\int_{B \times B} f(y,\,z)\d y \,\d z\right]^{p/2} \nonumber\\
&\lesssim  L_{\sigma}^p p^\frac{p(1-2 \beta \lambda)}{2} b_t(\theta) (2\theta)^\frac{p(2\beta \lambda-1)}{2} \Gamma (1-2\beta \lambda)^\frac{p}{2} \left[\int_{B \times B} f(y,\,z)\d y \,\d z\right]^{p/2},
\end{align}
where $b_t(\theta)=\sup_{0<s<t}\sup_{\text{z}\in B}e^{-\theta s} \E|u_s^{(\gamma)}(\text{z})-u_s^{(\beta)}(\text{z})|^p$. For $t > 0$, we fix $\theta > 0$ sufficiently large so that
\begin{align} \label{JNA-c21}
J_3 \lesssim \frac{1}{2} b_t(\theta).
\end{align}

Next we work on $J_2$.  Using Proposition \ref{JNA-c42} below and the assumption on $\sigma$, we obtain
\begin{align} \label{JNA-c20-b}
J_2&:=e^{-\theta t} \bigg[\int_0^t\int_{B\times B}|G_B^{(\gamma)}(t-s,\,x,\,y)-G_B^{(\beta)}(t-s,\,x,\,y)||G_B^{(\gamma)}(t-s,\,x,\,z)-G_B^{(\beta)}(t-s,\,x,\,z)|\times \nonumber\\
&~~~~~~~~~~\bigg[[\E|\sigma(u_s^{(\gamma)}(y))\sigma(u_s^{(\gamma)}(z))|^{p/2}]^{2/p} \,f(y,\,z)\, \d y\,\d z \,\d s\bigg]^{p/2} \nonumber\\
&\lesssim \sup_{t>0,\,x\in B}e^{-\theta t}\E|u_s^{(\gamma)}(x)|^p\bigg[\int_0^t\int_{B\times B}e^{-2\theta (t-s)/p}|G_B^{(\gamma)}(t-s,\,x,\,y)-G_B^{(\beta)}(t-s,\,x,\,y)|\times \nonumber\\
&~~~~~~~~~~|G_B^{(\gamma)}(t-s,\,x,\,z)-G_B^{(\beta)}(t-s,\,x,\,z)| \,f(y,\,z)\, \d y\,\d z \,\d s\bigg ]^{p/2}.
\end{align}
We combine the above estimates and obtain the following.
\begin{align*}
\sup_{x\in B} e^{-\theta s} &\E|u_s^{(\gamma)}(x)-u_s^{(\beta)}(x)|^p \lesssim  e^{-\theta t} |(\mathcal{G}_B^{(\gamma)} u)_t(x)-(\mathcal{G}_B^{(\beta)}u)_t(x)|^p + \frac{1}{2} b_t(\theta) +\\
& \sup_{t>0,\,x\in B}e^{-\theta t}\E|u_s^{(\gamma)}(x)|^p\bigg[\int_0^t\int_{B\times B}e^{-2\theta (t-s)/p}|G_B^{(\gamma)}(t-s,\,x,\,y)-G_B^{(\beta)}(t-s,\,x,\,y)| \times \nonumber \\ &~~~~~~~~~~|G_B^{(\gamma)}(t-s,\,x,\,z)-G_B^{(\beta)}(t-s,\,x,\,z)|\,f(y,\,z)\,\d s\, \d y\,\d z \bigg]^{p/2}.
\end{align*}
Since Lemma \ref{JNA-27}  and Proposition \ref{nane-tuan-2022-b} allow us to use dominated convergence theorem,  taking $\gamma\rightarrow \beta$ the result of the theorem follows.
\end{proof}

The following proposition is used in the proof of Theorem \ref{Thrm-10}. We use the bounds of $G_B^{(\beta)}(t,\,x,\,y)$ to show the proposition holds:
\begin{proposition} \label{JNA-c42}
Assume that $\{\varphi_n\}_{n\geq1}$ are uniformly bounded by a constant $C(B)$. Suppose that $d<\frac{\alpha}{2\beta}$, and $\int_{B \times B} f(x,\,y)\, \d x \,\d y <\infty$. For some $\theta$, the supremum on the $p^{th}$ moment of the solution $u_s^{(\beta)}(x)$ to equation  \eqref{JNA-8-b} given by
\begin{equation}
\sup_{0\leq s\leq t,\,x\in B}e^{-\theta s}\E|u_s^{(\beta)}(x)|^p
\end{equation}
is bounded above by a constant independent of $\beta$.
\end{proposition}
\begin{proof} We use the mild formulation in equation \eqref{JNA-9-b} and follow similar computations to those used in the above proof. Now consider
\begin{equation}\label{JNA-46-c}
u_t^{(\beta)}(x)=
(\mathcal{G}_B^{(\beta)} u_0)_t(\text{x})+ \lambda \int_B\int_0^t G_B^{(\beta)}(t-s,\,\text{x},\,\text{y})\sigma(u_s^{(\beta)}(\text{y}))F(\d s\,,\d \text{y}),
\end{equation}
Now applying the inequality $(a+b)^p \leq 2^p(a^p\,+b^p)$ for any $a,b \geq 0$ and using Burkholder-Davis-Gundy inequality, the $p^{th}$ moment of \eqref{JNA-46-c} becomes
\begin{align}\label{JNA-47-c}
\E|u_t^{(\beta)}(x)|^p &\leq
2^p|(\mathcal{G}_B^{(\beta)} u_0)_t(\text{x})|^p+ (2\lambda)^p  \bigg [\int_{B\times B} \int_0^t |G_B^{(\beta)}(t-s,\,\text{x},\,\text{y})G_B^{(\beta)}(t-s,\,\text{x},\,\text{z})| \times \nonumber \\
&~~~~~~~~~~~~~~ \bigg [\E|\sigma(u_s^{(\beta)}(\text{y}))\sigma(u_s^{(\beta)}(\text{z}))|^{p/2} \bigg]^{2/p} \, f(y,\,z)\,\d s\,\d \text{y} \,\d \text{z} \bigg]^{\frac{p}{2}},
\end{align}
\begin{align}\label{JNA-48-c}
e^{-\theta t}\E|u_t^{(\beta)}(x)|^p &\lesssim
e^{-\theta t}|(\mathcal{G}_B^{(\beta)} u_0)_t(\text{x})|^p+ e^{-\theta t} \bigg [\int_{B\times B} \int_0^t |G_B^{(\beta)}(t-s,\,\text{x},\,\text{y})G_B^{(\beta)}(t-s,\,\text{x},\,\text{z})| \times \nonumber \\
&~~~~~~~~~~~~~~ \bigg [\E|\sigma(u_s^{(\beta)}(\text{y}))\sigma(u_s^{(\beta)}(\text{z}))|^{p/2} \bigg]^{2/p} \, f(y,\,z)\,\d s\,\d \text{y} \,\d \text{z} \bigg]^{\frac{p}{2}} \nonumber \\
&:=J_1 + J_2.
\end{align}
Since $\sigma$ is globally Lipschitz, using Lemma \ref{JNA-45}, and Lemma  \ref{nane-tuan}(a), the second term becomes
\begin{align} \label{JNA-49-c}
J_2 &:=e^{-\theta t} \bigg [\int_{B\times B} \int_0^t |G_B^{(\beta)}(t-s,\,\text{x},\,\text{y})G_B^{(\beta)}(t-s,\,\text{x},\,\text{z})| \bigg [\E|\sigma(u_s^{(\beta)}(\text{y}))\sigma(u_s^{(\beta)}(\text{z}))|^{p/2} \bigg]^{2/p} \, f(y,\,z)\,\d s\,\d \text{y} \,\d \text{z} \bigg]^{\frac{p}{2}},
 \nonumber\\
&\lesssim L_{\sigma}^p \bigg[\int_{B \times B} \int_0^t G_B^{(\beta)}(t-s,\,x,\,y)G_B^{(\beta)}(t-s,\,x,\,z)  \times \nonumber\\
& ~~~~~~~~~~e^{\frac{-2\theta t}{p}}[(\E[(u_s^{(\beta)}(y))^{p})^{\frac{1}{p}} (\E[(u_s^{(\beta)}(z))]^{p})^{\frac{1}{p}} \,f(y,\,z)\,\d s\,\d y \,\d z \bigg]^{p/2}.\nonumber\\
&\lesssim  L_{\sigma}^p b_t(\theta) \left[\int_{B \times B} \int_0^t G_B^{(\beta)}(t-s,\,x,\,y)G_B^{(\beta)}(t-s,\,x,\,z)\, e^{\frac{-2\theta (t-s)}{p}}\,f(y,\,z)\,\d s\,\d y \,\d z\right]^{p/2} \nonumber\\
&\lesssim  L_{\sigma}^p b_t(\theta) \left[\int_0^t (t-s)^{-2\beta \lambda}  e^{\frac{-2\theta (t-s)}{p}}\,\d s\,\int_{B \times B} f(y,\,z)\d y \,\d z\right]^{p/2} \nonumber\\
&\lesssim  L_{\sigma}^p b_t(\theta) \left[\int_0^\infty s^{-2\beta \lambda}  e^{\frac{-2\theta s}{p}}\,\d s\,\int_{B \times B} f(y,\,z)\d y \,\d z\right]^{p/2} \nonumber\\
&\lesssim  L_{\sigma}^p p^{1-2 \beta \lambda} b_t(\theta) \left[\int_0^\infty \frac{e^{-2\theta r}}{r^{2\beta \lambda} } \,\d s\,\int_{B \times B} f(y,\,z)\d y \,\d z\right]^{p/2} \nonumber\\
&\lesssim  L_{\sigma}^p p^\frac{p(1-2 \beta \lambda)}{2} b_t(\theta) (2\theta)^\frac{p(2\beta \lambda-1)}{2} \Gamma (1-2\beta \lambda)^\frac{p}{2} \left[\int_{B \times B} f(y,\,z)\d y \,\d z\right]^{p/2},
\end{align}
where $b_t(\theta)=\sup_{0<s<t}\sup_{\text{z}\in B}e^{-\theta s} \E|u_s^{(\beta)}(\text{z}|^p$. For $t > 0$, we fix $\theta > 0$ sufficiently large so that
\begin{align} \label{JNA-50-c}
J_2 \lesssim \frac{1}{2} b_t(\theta).
\end{align}
Thus, we have
\begin{align}\label{JNA-51-c}
e^{-\theta t}\E|u_t^{(\beta)}(x)|^p &\lesssim 1+\frac{1}{2} b_t(\theta).
\end{align}
Taking the supremum of the left side of \eqref{JNA-51-c} the result follows.
\end{proof}

\end{document}